\documentclass[10pt]{amsart}

\usepackage{palatino}
\usepackage{eucal}
\usepackage{amssymb,url,amscd,amsfonts}

\theoremstyle{plain}
\newtheorem{theorem}{Theorem}

\newtheorem{lemma}[theorem]{Lemma}
\newtheorem{proposition}[theorem]{Proposition}

\numberwithin{equation}{section} 
\newtheorem{remark}{Remark}

\begin{document}

\title[The Dirichlet Problem for Curvature Equations]
{The Dirichlet Problem for Curvature Equations in Riemannian Manifolds}

\thanks{Work partially supported by CAPES and CNPq.}

\author[J. H. S. de Lira]{Jorge H. S. de Lira}
\address{Departamento de Matem\'atica \\ Universidade Federal do Cear\'a\\ 
Campus do Pici, Bloco 914\\ Fortaleza, Cear\'a\\ Brazil\\ 60455-760}
\email{jorge.lira@pq.cnpq.br}

\author[F. F. Cruz]{Fl\'avio F. Cruz}
\address{Departamento de Matem\'atica\\ Universidade Regional do Cariri\\ Campus Crajubar \\ 
Juazeiro do Norte, Cear\'a\\ Brazil\\ 63041-141}
\email{flavio.franca@urca.br}


\begin{abstract}
We prove the existence of classical solutions to the Dirichlet problem for a class of 
fully nonlinear elliptic equations of curvature type on
Riemannian manifolds. We also derive new second derivative boundary estimates
which allows us to extend some of the existence theorems of Caffarelli, Nirenberg and
Spruck \cite{CNSV} and Ivochkina, Trudinger and Lin \cite{IVO2}, \cite{IVO-LIN-TRU}, \cite{LIN-TRU}
to more general curvature functions under mild conditions on the geometry of the domain.

\end{abstract}

\maketitle


\section{Introduction}
\label{section1}

The aim of this article is to study the classical Dirichlet problem for equations of prescribed
curvature of the form
\begin{equation}
\label{P} F[u]=f(\kappa [u] )= \Psi (x,u)
\end{equation}
defined on a smooth Riemannian manifold $(M^n, \sigma), \, n\geq 2,$ 
where $\kappa [u]$ is the vector in $\mathbb{R}^n$ whose components $\kappa_1, \ldots , \kappa_n$ 
are the principal curvatures of the graph 
$\Sigma=\{(x, u(x)), x\in \Omega\}\subset M\times \mathbb{R}$ 
of a function $u$ defined on a bounded  domain $\Omega\subset M,$ $\Psi$ is a prescribed positive 
function on $\overline{\Omega}\times 
\mathbb{R}$ and $f$ is a general curvature function in a sense that it will made precise later. 
The main examples of general curvature functions
are given by the $k$-th root of the higher order mean curvatures
\begin{equation}
\label{SK}
S_k(\kappa)= \sum_{i_{1}<\ldots <i_{k}} \kappa_{i_1}\cdots
\kappa_{i_k}
\end{equation}
and the $(k-l)$-th root of their
quotients $S_{k,l}=S_k/S_l, \, 0\leq l< k\leq n.$ The mean, scalar, Gauss and harmonic curvatures
correspond to the special cases $k=1,2, n$ in (\ref{SK}) and $k=n, l=n-1$ for the quotients, respectively.

The classical Dirichlet problem associated to equation (\ref{P}) has been extensively studied
(see for instance  \cite{CNSV}, \cite{GUAN-LI}, \cite{GUAN-SPRUCK-2}, \cite{GUAN-SPRUCK-3}, 
\cite{IVO2}, \cite{IVO-LIN-TRU},  \cite{IVO-TOMI}, \cite{SHENG-URBAS-WANG}, \cite{SPRUCK} and \cite{TRU1}).
For domains in the Euclidean space, the first breakthroughs about the solvability of the Dirichlet problem 
\begin{align}
\begin{split}
\label{PD-EQ} F[u]=f(\kappa[u] )&=\Psi \quad\textrm{in}\quad \Omega \\ 
u&=\varphi \quad\textrm{on}\quad\partial \Omega,
\end{split}
\end{align}
were due to Caffarelli, Nirenberg and Spruck 
\cite{CNSV} for general curvature functions
and Ivochkina \cite{IVO1} for the particular cases of higher order mean curvatures (\ref{SK}). 
These authors established the solvability of (\ref{PD-EQ}) for 
the case of uniformly convex domains and zero boundary values. 
In \cite{IVO2} Ivochkina extended her approach to embrace general
boundary values and the more general $k$-convex domains, extending
the result of J. Serrin \cite{SERRIN} on the quasilinear case
corresponding to the mean curvature. 
Despite the cases of higher order mean curvatures be covered by the generality 
allowed in the theorem of Caffarelli, Nirenberg and Spruck \cite{CNSV} 
their result makes use of a strong  technical assumption on the curvature functions,
which precludes the case of quotients
$f=\left(S_{k,l}\right)^{1/(k-l)}, \, 0\leq l< k\leq n.$
However, the weak or viscosity solution approach by Trudinger \cite{TRU1} 
is sufficiently general in what concerns the curvatures functions 
so that it includes those quotient curvature functions. Also in \cite{TRU1} Trudinger establishes the 
existence theorems for Lipschitz solutions of (\ref{PD-EQ}) for general
boundary values and domains subjects to natural geometric restrictions. 
In  the subsequent articles \cite{IVO-LIN-TRU}  and 
\cite{LIN-TRU} Ivochkina, Lin and Trudinger extended  the approach used by
Ivochkina in \cite{IVO2} to the cases of quotients,
thereby obtaining globally smooth solutions.
Their approach makes use of highly specific properties 
of these particular curvature functions. Finally, Sheng, Urbas and Wang \cite{SHENG-URBAS-WANG}
derive an interior curvature bound for solutions of (\ref{PD-EQ}) which permits
to improve the existence results of Trundinger to yield locally smooth solutions.

For domains in a general Riemannian manifolds, the cases of mean, Gauss and harmonic 
curvatures have been extensively studied, as can be seen in 
\cite{ALIAS-DAJCZER}, \cite{ATALLAH-ZUILY},
\cite{DHL}, \cite{GUAN-3}, \cite{GUAN-LI}, \cite{GUAN-SPRUCK-1}, \cite{HRS}, 
\cite{IVO-TOMI}, \cite{THOMAS1}, \cite{THOMAS2} and \cite{SPRUCK}. 
Nevertheless, to the best of our knowledge there are no results on the existence of solutions for
(\ref{PD-EQ}) when $f$ is a general curvature function  (even for the case of higher order mean curvature) 
and $\Omega$ is a domain in a Riemannian manifold.

In this paper we deal with the Dirichlet problem (\ref{PD-EQ}) for general curvature functions
in the general setting of domains 
$\Omega$ in a smooth Riemannian manifold $(M^n, \sigma)$ under natural geometric conditions. 
We also derive a new boundary second derivative estimate which allows us to
improve the existence results in \cite{SHENG-URBAS-WANG}, yielding globally smooth solutions.
This is more precisely stated in what follows.

As in \cite{CNSV} and \cite{SHENG-URBAS-WANG}, we assume that $f\in C^2(\Gamma)\cap C^0(\overline
\Gamma)$ is a symmetric function defined in 
an open, convex, symmetric cone $\Gamma 
\subset\mathbb{R}^n$ with vertex at the origin and containing the positive cone 
$\Gamma^+=\{\kappa\in\mathbb{R}^n : \, \textrm{each component} \, \kappa_i>0\}$.
We suppose that $f$  satisfies the fundamental structure conditions
\begin{equation}
\label{elipticidade} f_i=\frac{\partial f}{\partial \kappa_i}>0
\end{equation}
and
\begin{equation}
\label{concavidade}
f \textrm{ is a concave function}.
\end{equation}
In addition, $f$ is assumed to satisfy the following more technical assumptions
\begin{align}
\label{soma} \sum f_i(\kappa)&\geq c_0>0 \\ 
\label{relacao_euler}\sum  f_i(\kappa)\kappa_i&\geq c_0>0 \\ 
\label{limite} \limsup_{\kappa\rightarrow\partial \Gamma }f(\kappa) &\leq \bar{\Psi}_0<\Psi_0\\
\label{ki-negativo} f_i(\kappa) &\geq c_0>0 \,\,\textrm{for any} \,\,
\kappa\in\Gamma \, \, \textrm{with} \, \, \kappa_i<0
\\\label{produto} (f_1 \cdots f_n)^{1/n}&\geq c_ 0
\end{align}
for $\kappa\in \Gamma_{\Psi}= \{\kappa\in \Gamma : \, \Psi_0 \leq f(\kappa)\leq \Psi_1 \} $ and 
a constant  $c_0$ depending on $ \Psi_0 $ and $\Psi_1,$ where  $\Psi_0=\inf\Psi$ and 
$\Psi_1=\sup \Psi.$  In this context, a function $u\in C^2(\overline\Omega)$ is 
called {\it admissible} if $\kappa [u]\in \Gamma$ at each point of its graph.
We point out that these conditions also appear in \cite{CNSV}, \cite{GUAN-2}
and \cite{TRU1}. It has been shown (see \cite{CNSIII}, \cite{CNSV}, \cite{GUAN-2}, 
\cite{YYLI}, \cite{LIN-TRU-2}  and \cite{TRU1}) that the (root of) higher
order mean curvatures and their quotients satisfy 
(\ref{elipticidade})-(\ref{produto}) on the appropriate cone $\Gamma.$ 

In  \cite{CNSV}  and \cite{GUAN-SPRUCK-2}, the hypotheses on the curvature functions 
include the requeriment that 
for every constant $C>0$ and every compact set $E$ in $\Gamma$ there is a number $R=R(C,E)$ such that
\begin{equation}
\label{INDESEJADA}
f(\kappa_1, \cdots ,\kappa_n+R)\geq C, \hspace*{.5cm} \forall\, \kappa\in E,
\end{equation}
which precludes the important examples of the quotients $f=\left(S_{k,l}\right)^{1/(k-l)}.$ 
More precisely, this condition is used in \cite{CNSV} and \cite{GUAN-SPRUCK-2} 
to estimate the double normal derivative of admissible solutions at the boundary.
In this paper we adapt a technique presented in \cite{TRU1} to cover the cases
where (\ref{INDESEJADA}) does not hold.

As in the papers \cite{CNSV}, \cite{SHENG-URBAS-WANG}, \cite{TRU1} we shall also
need conditions on the boundary $\partial\Omega$ to ensure the attainment of Dirichlet boundary
conditions. We assume that $\Omega$ is a domain with $C^2$ boundary and that the principal curvatures 
$\kappa'=(\kappa'_1,\cdots,\kappa'_{n-1})$ of $\partial\Omega$  satisfy
\begin{equation}
\label{SERRIN_CONDITION}
 f(\kappa'(y),0) \geq \Psi(y , \varphi(y)),\hspace*{.5cm} \forall\, y\in\partial\Omega.
\end{equation}
We note that (\ref{SERRIN_CONDITION}) is the natural extension of the Serrin
condition for the mean curvature case \cite{SERRIN} and it implies that
\begin{equation}
\label{f-CONVEXIDADE}
(\kappa'_1,\cdots,\kappa'_{n-1}, 0)\in \Gamma.
\end{equation}
We can now formulate our main existence theorem for the Dirichlet problem. Let
$(M^n, \sigma), \, n\geq 2,$ be a complete orientable Riemannian manifold and
$\Omega$ a connected bounded domain in $M.$
\begin{theorem}
\label{teorema1}
Let $f\in C^2(\Gamma)\cap C^0(\overline\Gamma)$ be a curvature function satisfying conditions 
(\ref{elipticidade})-(\ref{produto}). Suppose that there exists $0<\alpha<1$ such 
that $\partial\Omega\in C^{4,\alpha}$ and
$\Psi \in C^{2,\alpha}(\overline\Omega\times\mathbb{R}).$  
Moreover, suppose that $\Psi$ satisfies  (\ref{SERRIN_CONDITION}) and 
that $\Psi>0$ and $\Psi_t\geq 0$ on $ 
\overline\Omega\times\mathbb{R} $. If there exists a locally strictly convex $C^2$ function in 
$\overline\Omega$ and there exists an admissible subsolution 
$\underline u\in C^{2}(\overline\Omega)$
of equation (\ref{P}),  then there exists a unique admissible solution 
$u \in C^{4,\alpha}(\overline{\Omega})$ of the Dirichlet problem  (\ref{PD-EQ})
for any given function $\varphi\in C^{4,\alpha}(\overline \Omega).$
\end{theorem}

The assumption on the existence of a strictly convex function in 
$C^{2}(\overline\Omega )$ arises naturally when the problem is treated in a 
general Riemannian manifold, as could be seen for instance in \cite{GERHARDT} and \cite{GUAN-2}.
Notice that when $M$ is the Euclidean space this condition is always satisfied.
Theorem \ref{teorema1} extends the result of Caffarelli, Nirenberg and Spruck 
presented in \cite{CNSV} to non-convex domains and general boundary values, without the
assumption (\ref{INDESEJADA}) and also improves the existence results of \cite{SHENG-URBAS-WANG}
and \cite{TRU1}  to yield globally smooth solutions for general 
boundary values.

The cases $f=\left(S_{n,l}\right)^{1/(n-l)}, \, l=1,\ldots, n-1,$ are omitted
from Theorem \ref{teorema1}, as the corresponding extension of the
Serrin condition (\ref{SERRIN_CONDITION}) would imply $\Psi=0$ on $\partial\Omega,$
contradicting the hypothesis on the positivity of $\Psi.$ 
However these cases are covered by Theorem \ref{teorema2}
that is presented below (see Remark \ref{quotient-rem} in Section \ref{section5}).

It is well known that conditions on the geometry of 
the boundary $\partial \Omega$ play a key role in 
the study of the solvability of the Dirichlet problem 
(\ref{PD-EQ}). Nevertheless, several authors (e.g., \cite{GUAN-1}, \cite{GUAN-3},
\cite{GUAN-2}, \cite{GUAN-LI}, \cite{GUAN-SPRUCK-1} and \cite{GUAN-SPRUCK-2}) 
had  replaced  geometric conditions 
on the boundary by assumptions on the 
existence of a subsolution satisfying the boundary condition.
For more details we refer the reader to \cite{CNSIII} and \cite{GUAN-2},  
where is shown the existence of a close relationship 
between the convexity of the boundary and the existence of such subsolutions. 
Therefore, it is natural to consider a version of Theorem \ref{teorema1} obtained
replacing the assumption on the geometry of the boundary $\partial\Omega$
by the assumption on the existence of a subsolution satisfying the boundary condition.
In this context we obtain the following result.
\begin{theorem}
\label{teorema2}
Suppose  that
$f\in C^2(\Gamma)\cap C^0(\overline\Gamma)$ 
satisfy (\ref{elipticidade})-(\ref{produto}) and that for some $0<\alpha<1$ it holds that
\begin{quotation}
\noindent (i) $\partial\Omega\in C^{4,\alpha}$ has 
nonnegative mean curvature;\\
(ii) there exists an admissible subsolution $\underline u \in C^{2}(\overline{\Omega})$ 
of equation (\ref{P}) such that $\underline u=\varphi$ on $\partial\Omega$ 
and $\underline u$ is locally strictly 
convex (up to the boundary) in a neighborhood of $\partial\Omega;$\\  
\noindent (iii) $\Psi\in C^{2,\alpha}(\overline\Omega\times\mathbb{R}),$ 
 $\Psi>0$ and $\Psi_t\geq 0$ on $ 
\overline\Omega\times\mathbb{R};$\\
(iv) there exists a locally strictly convex $C^2$ function in $\overline\Omega.$ 
\end{quotation} 
Then there exists a unique admissible solution $u \in C^{4,\alpha}(\overline{\Omega})$ of the 
Dirichlet problem (\ref{PD-EQ}) for any given function $\varphi\in C^{4,\alpha}(\overline{\Omega}).$
\end{theorem}

Theorem \ref{teorema2} extends to domains in a general 
Riemannian manifold the result obtained by 
Guan and Spruck in \cite{GUAN-SPRUCK-2}.
In addition, Theorem \ref{teorema2} embraces 
a class of curvature functions larger than the one considered
in \cite{GUAN-SPRUCK-2}, including the higher order mean curvatures and their quotients and,
more generally, curvature functions that are defined in a general 
cone $\Gamma,$ not necessarily  being the positive cone $\Gamma^+.$
We point out that the requirement on the mean curvature of $\partial\Omega$ in (i) 
is necessary because of the generality of the cone $\Gamma$ allowed in the Theorem \ref{teorema2}. 
This kind of assumption was already made before in earlier contributions to the subject, 
see for instance  \cite{GUAN-SPRUCK-4}.

Following \cite{CNSV}, \cite{GT}, \cite{GUAN-SPRUCK-2}, \cite{IVO1} and 
\cite{IVO-LIN-TRU}, the proofs of the above existence theorems utilize the method of continuity which 
reduces the problem of the existence 
to the establishment of {\it a priori} estimates for a related family of Dirichlet problems 
in the H\"{o}lder space
$C^{2,\beta}(\overline{\Omega})$ for some $\beta>0.$ 
Here we will establish $C^2$ {\it a priori} estimates. 
H\"{o}lder bounds for the second order derivatives then follows 
from the Evans-Krylov theory (see for example \cite{GT} and \cite{KRYLOV}) while higher 
order estimates follows from the classical Schauder theory. The uniqueness in 
Theorems \ref{teorema1} and \ref{teorema2} is a consequence of the comparison principle.
As is pointed out in \cite{CNSV}, \cite{IVO-LIN-TRU}, \cite{IVO-TOMI} 
and \cite{LIN-TRU}, the crucial estimates are those of the second derivatives
on the boundary $\partial\Omega.$ 
In this work, the use of a new barrier allows us to obtain
estimates for mixed tangential normal derivatives at the boundary for
solutions of (\ref{PD-EQ}) for general curvature equations, which is new even in the Euclidean case.
This is one of the main achievements of this work.
We establish the double normal second derivative estimates at the boundary 
by extending the techniques of \cite{CNSIII},  \cite{GUAN-2} and \cite{TRU-2}
to equations of curvature type.

This article is organized as follows:  In Section \ref{section2} we list some basic formulas 
which are needed later. In Sections \ref{section3}-\ref{section6} we deal with the {\it a priori} estimates 
for prospective solutions of (\ref{PD-EQ}). The height and boundary gradient estimates 
are derived in Section \ref{section3}. Section \ref{section4} is devoted to the proof of 
the global gradient estimates. 
In Section \ref{section5} {\it a priori}
bounds for the second order derivatives on the boundary  are established while
in Section \ref{section6} we show how to estimate the second derivatives of solutions 
given boundary estimates for them.
Finally, in Section \ref{section7} we complete the proofs of Theorems \ref{teorema1} and 
\ref{teorema2} using the continuity method based on the previously
established estimates.

Finally we would like to mention that quite recently Profs. H. Blaine Lawson Jr.  and F. Reese Harvey have developed a very powerful new method for handling with geometric fully nonlinear equations of the type we are considering here.  To adapt the ideas presented in, e.g.,  \cite{HL} is without doubt a very promising way to attack prescribed curvature problem  (\ref{PD-EQ}) in its full generality.


\section{Preliminaries}
\label{section2}

Let $(M^n,\sigma)$ be a complete Riemannian manifold.
We consider the product  manifold $\bar M = M \times 
\mathbb{R}$ endowed with the product metric. The  Riemannian connections 
in $\bar M$ and $M$ will be denoted
respectively by $\bar\nabla$ and $\nabla$. The curvature tensors in
$\bar M$ and $ M$ will be represented by $\bar R$ and $R$,
respectively. The convention used here for the curvature tensor is
\[
R(U,V)W=\nabla_V\nabla_U W-\nabla_U\nabla_V W+\nabla_{[U,V]}W.
\]
In terms of a coordinate system $(x^i)$ we write
\[
R_{ijkl}=\sigma\left(R\left(\frac{\partial}{\partial x^i}, 
\frac{\partial}{\partial x^j} \right) \frac{\partial}{\partial x^k} , 
\frac{\partial}{\partial x^l}   \right).
\]
With this convention, the Ricci identity for the derivatives of a smooth 
function $u$ is given by
\begin{equation}
\label{ricci}
u_{i;jk}=u_{i;kj}+R_{ilkj}u^l.
\end{equation}

Let $\Omega$ be a bounded domain in $M.$
Given a differentiable function $u:\Omega\to \mathbb{R},$ its graph is defined as
the hypersurface $\Sigma$ parameterized by $Y(x)=(x,u(x))$ with
$x\in \Omega$. This graph is diffeomorphic with $\Omega$ and may be globally
oriented by an unit normal vector field $N$ for which it holds that
$\langle N,\partial_t\rangle >0$, where $\partial_t$ denotes the usual coordinate vector 
field in $\mathbb{R}$. With respect to this orientation,
the second fundamental form in $\Sigma$ is by definition the
symmetric tensor field $b=-\langle d N, d X\rangle$. We will
denote by $\nabla'$ the connection of $\Sigma.$

The unit vector field
\begin{equation}\label{normalvect1}
N = \frac{1}{W}\big( \partial_t-\nabla u \big)
\end{equation}
is normal to $\Sigma$, where
\begin{equation} \label{wz}
W = \sqrt{1 + |\nabla u|^2}.
\end{equation}
Here, $|\nabla u|^2=u^iu_i$ is the squared norm of $\nabla u$. 
The induced metric in  $\Sigma$ has components
\begin{equation}
\label{gij1} g_{ij} = \langle Y_i, Y_j\rangle = \sigma_{ij} +
u_iu_j
\end{equation}
and its inverse has components given by
\begin{equation}
\label{gij2} g^{ij}=\sigma^{ij}-\frac{1}{W^2}u^iu^j.
\end{equation}
We easily verify that the components $(a_{ij})$ of the second fundamental form of
$\Sigma$ are  determined by
\begin{eqnarray*}
\label{bij1} a_{ij} =\langle \bar\nabla_{Y_j}Y_i,N\rangle=
\frac{1}{W}u_{i;j}
\end{eqnarray*}
where $u_{i;j}$ are the components of the Hessian $\nabla^2u$ of $u$ in $\Omega.$
Therefore the components $a_i^j$ of the Weingarten map $A^{\Sigma}$ of
the graph $\Sigma$ are given by
\begin{equation}
\label{A-I-J}
a_i^j= g^{jk}a_{ki}= \frac{1}{W}\left( \sigma^{jk}-\frac{1}{W^2}u^ju^k\right)
u_{k;i}.
\end{equation}
Above and throughout the text we made use of the  Einstein summation convention.

For our purposes it is crucial to know the rules of commutation involving the covariant derivatives,
the second fundamental form of a hypersurface and the curvature of the ambient. In this sense,
the Gauss and Codazzi equations will play a fundamental role. They are, respectively,
\begin{align}
\label{GAUSS}
R'_{ijkl}&=\bar{R}_{ijkl}+a_{ik}a_{jl}-a_{il}a_{jk}
\\
\label{CODAZZI}
a_{ij;k}&=a_{ik;j}+\bar{R}_{i0jk}
\end{align}
where the index $0$ indicates coordinate components of the normal vector $N$ and $R'$ is the 
Riemann tensor of $\Sigma.$ We note that $a_{ij;k}$ indicates the 
componentes of the tensor $\nabla' b,$ obtained
by differentiating covariantly the second fundamental form $b$
of $\Sigma$ with respect to the metric $g.$
The following identity for commuting  second derivatives of the
second fundamental form will be quite useful. It was first found by Simons
in \cite{SIMONS} and in our notation it assumes the form
\begin{align}
\begin{split}
\label{SIMONS-FORMULA}
a_{ij;kl}= \, & a_{kl;ji}+a_{kl}a_i^ma_{jm}-a_{ik}a_j^ma_{lm} +a_{lj}a_i^ma_{km}-a_{ij}a_l^ma_{km}
 \\&+\bar{R}_{likm}a^m_j +\bar{R}_{lijm}a_k^m-\bar{R}_{mjik}a_l^m -\bar{R}_{0i0j}a_{kl}
 +\bar{R}_{0l0k}a_{ij}\\&-\bar{R}_{mkjl}a_i^m-\bar\nabla_l\bar{R}_{0jik}-\bar\nabla_i\bar{R}_{0kjl}.
\end{split}
\end{align}

Let $\mathcal{S}$ be the space of all symmetric covariant tensors of rank two defined in 
the Riemannian manifold $(\Sigma, g)$  and $\mathcal{S}_{\Gamma}$ 
be the open subset of those symmetric tensors $a\in\mathcal S$ 
for which the eigenvalues (with respect to the metric $g)$ are
contained in $\Gamma.$ Then we can define the mapping
$F: \mathcal{S}_\Gamma\longrightarrow\mathbb R$ by setting
\[
F(a)=f\big(\lambda(a)\big),
\]
where $\lambda(a)=(\lambda_1, \cdots, \lambda_n)$ are the eigenvalues of $a.$
The mapping $F$ is as smooth as $f.$ 
Furthermore, $F$ can be viewed as depending solely on the mixed tensor $a^{\sharp},$ 
obtained by raising one index of the given symmetric covariant $2$-tensor $a,$ 
as well as depending on the pair of covariant tensors $(a, g),$ 
\[
F(a^\sharp)= F(a,g).
\]
In terms of components, in an arbitrary coordinate system we have
\[
F\big(a_i^j\big)=F\big(a_{ij}, g_{ij}\big)
\]
with $a_i^j=g^{jk}a_{ki}.$ We denote the first derivatives of $F$ by
\[
F^{ij}=\frac{\partial F}{\partial a_{ij}} \hspace*{.5cm} \textrm{and}
\hspace*{.5cm} F_i^j=\frac{\partial F}{\partial a_j^i},
\]
and the second derivatives are indicated by 
\[
F^{ij,kl}=\frac{\partial^2 F}{\partial a_{ij}\partial a_{kl}}.
\]
Hence $F^{ij}$ are the components of a symmetric covariant tensor,
while $F_i^j$ defines a mixed tensor which is contravariant with respect
to the index $j$ and covariant with respect to the index $i.$

As in \cite{IVO2}, we extend the cone $\Gamma$ to the space of symmetric matrices 
of order $n,$ which we denote (also) by $\mathcal{S}$. Namely, for $
p\in \mathbb R^n$, let us define
\[
\Gamma(p)=\{r\in \mathcal{S}\,: \, \lambda(p,r)\in \Gamma\},
\] 
where $\lambda(p,r)$ denotes the eigenvalues of the matrix 
$A(p,r)=g^{-1}(p)r$ given by
\begin{align}
\label{matriz-A}
A(p,r)=\frac{1}{\sqrt{1+|p|^2}}\left(I-\frac{p\otimes p}{1+|p|^2}\right)r,
\end{align}
(the eigenvalues computed with respect to the Euclidean inner product). 
$A(p,r)$ is obtained from the matrix of the Weingarten map
with $(p,r)$ in place of $(\nabla u, \nabla^2 u)$ and $\delta^{ij}$
in place of $\sigma^{ij}.$ We note that the eigenvalues of $A(p,r)$
are the eigenvalues of $r$ (unless the $1/\sqrt{1+|p|^2}$ factor)
with respect to the inner product given by
the matrix $g=I+p\otimes p.$ In this setting it
is convenient to introduce the notation 
\[
G(p,r)=F\big(A(p,r)\big)= f\big(\lambda(p,r)\big).
\]
Hence,  as in \cite{CNSV} and \cite{GUAN-SPRUCK-2} we may write 
equation (\ref{P}) in the form
\begin{equation}
\label{P2-1}
F[u]=f(\kappa[u])=G(\nabla u , \nabla^2 u)=\Psi(x,u).
\end{equation}
In particular, if we denote
\[
G^{ij}=\frac{\partial G}{\partial r_{ij}} \hspace*{.5cm} \textrm{and}
\hspace*{.5cm} G^{ij,kl}=\frac{\partial^2 G}{\partial r_{ij}\partial r_{kl}},
\]
we obtain
\[
G^{ij} =\frac{1}{W}F^{ij}
\hspace*{.5cm} \textrm{and}
\hspace*{.5cm}
G^{ij,kl}= \frac{1}{W^2}F^{ij,kl}.
\]
The derivatives of the mapping $F$ may be easily computed if we assume that 
the matrix $\big(a_{ij}\big)$ is diagonal with respect to the 
metric $\big ( g_{ij}\big),$ as is shown in the following lemma.

\begin{lemma}
\label{derivada-general-f} Let 
$\{e_i\}_{i=1}^n$  be a local 
orthonormal (with respect to the metric $(g_{ij})$ in $\Sigma$) basis of eigenvectors
for  $a\in\mathcal S_\Gamma$  with corresponding eigenvalues $\lambda_i.$ Then, in terms of this basis
the matrix $(F^{ij})$ is also diagonal with eigenvalues $f_i= \frac{\partial f}
{\partial\lambda_i}$.
Moreover, $F$ is concave and its second derivatives are given by
\begin{equation}
F^{ij,kl}\eta_{ij}\eta_{kl}=\sum_{k,l}f_{kl}\eta_{kk}\eta_{ll}+\sum_{k\neq
l} \frac{f_k-f_l}{\lambda_k-\lambda_l}\eta_{kl}^2,
\end{equation}
for any $(\eta_{ij})\in\mathcal S.$ Finally we have
\begin{equation}
\frac{f_i-f_j}{\lambda_i-\lambda_j}\leq 0.
\end{equation}
These expressions must be interpreted as limits in the case of
principal curvatures with multiplicity greater than one.
\end{lemma}

It follows from the above lemma that, under condition (\ref{elipticidade}),
equation (\ref{P2-1}) is elliptic, i.e., the matrix $G^{ij}(p,r)$
is positive-definite for any $r\in\Gamma(p).$ Moreover, under condition (\ref{concavidade})
the restriction of the function $G(p,\cdot)$ to the open set 
$\Gamma(p)$ is a concave function. We point out that since $1/W$ and $1$ are
respectively the lowest and the largest eigenvalues of $g^{ij}$ it holds that
\begin{equation}
\frac{1}{W^3} F_i^j\delta^i_j\leq G^{ij}\delta_{ij}\leq \frac{1}{W}F_i^j\delta_j^i.
\end{equation}

Now we analyze some consequences of the conditions 
(\ref{elipticidade})-(\ref{limite}). First we note that 
the concavity condition implies 
\begin{equation}
\label{upper-f}
\sum_i f_i(\kappa)\kappa_i\leq f
\end{equation}
and we also may prove using assumptions (\ref{soma})-(\ref{limite}) and following \cite{CNSIII}
that
\begin{equation}
\label{gamma-longe-zero}
\sum_i \kappa_i\geq \delta >0,
\end{equation}
for any $\kappa\in \Gamma$ that satisfies $f(\kappa)\geq \Psi_0.$
This geometric fact implies that {\it upper bounds}
for the principal curvatures of the graph of an admissible solution immediately ensure
{\it lower bounds} for these curvatures.

Now we will derive a lemma that gives a 
useful formula involving the second and third derivatives
of prospective solutions to the problem (\ref{PD-EQ}).

\begin{lemma}
\label{lema-formula-psi-k}
Let $u$ be a solution of equation (\ref{P2-1}). The derivatives of 
$u$ satisfy the formula
\begin{align}
\begin{split}
\label{formula-psi-k}
G^{ij}u_{k;ij}=& WG^{ij}a_j^lu_{k;i}u_l+WG^{ij}a_j^lu_{k;l}u_i+
\frac{1}{W}G^{jl}a_{jl} u^i u_{i;k} \\ &-
G^{ij}R_{iljk}u^l+\Psi_k + \Psi_t u_{k}.
\end{split}
\end{align}
\end{lemma}
\begin{proof}
Differentiating covariantly equation (\ref{P2-1}) in the $k$-th direction
with respect to the metric $\sigma$ of $M$  we obtain
\begin{align}
\label{CONTA-PSI-K-1}
\Psi_k+\Psi_t u_k  &= \frac{\partial G}{\partial u_{i;j}}u_{i;jk} +
\frac{\partial G}{\partial u_{i}}u_{i;k}= G^{ij}u_{i;jk} +
G^iu_{i;k}.
\end{align}
From $F(a_i^j[u])=G(\nabla u, \nabla^2 u)$ we calculate
\begin{align*}
G^i &=\frac{\partial G}{\partial u_{i}}=\frac{\partial F}{\partial
a_r^s}\frac{\partial a_r^s}{\partial u_{i}}=
F_s^r\frac{\partial }{\partial u_{i}}\left( \frac{1}{W}g^{sl}u_{l;r}\right)
\\ &= F_s^r g^{sl}u_{l;r}\frac{\partial }{\partial u_{i}}\left( \frac{1}
{W}\right)+ \frac{1}{W} F_s^r\frac{\partial }{\partial u_{i}}\left( g^{sl}\right)u_{l;r}.
\end{align*}
We compute
\[
F_s^r g^{sl}u_{l;r}\frac{\partial }{\partial u_{i}}\left( \frac{1}
{W}\right)=-\frac{u^i}{W^3}F_s^r g^{sl}u_{l;r}=-\frac{1}{W} G^{rs}a_{rs}u^i
\]
and
\begin{align*}
\frac{1}{W} F_s^r\frac{\partial }{\partial u_{i}}\left( g^{sl}\right)u_{l;r} 
&= G^{rp}g_{sp}\frac{\partial }{\partial u_{i}}\left( g^{sl}\right)u_{l;r}
\\ &= -WG^{ij}a_j^lu_l-WG^{lj}a_l^iu_j,
\end{align*}
where we have used the expression
\[
g_{sp}\frac{\partial g^{sl}}{\partial u_i}=-g^{sl}
(\delta_{is}u_p+u_s\delta_{ip})=-(\delta_{ip}g^{sl}u_s+g^{il}u_p).
\]
It follows that
\begin{align*}
G^i=-\frac{1}{W} G^{rs}a_{rs}u^i-WG^{ij}a_j^lu_l-WG^{lj}a_l^iu_j.
\end{align*}
Replacing these relations into (\ref{CONTA-PSI-K-1}) we obtain
\begin{align*}
\Psi_k+ \Psi_t u_k = G^{ij}u_{i;jk} -\frac{1}{W} G^{rs}a_{rs}u^iu_{i;k}
-WG^{ij}a_j^lu_lu_{i;k}-WG^{lj}a_l^iu_ju_{i;k}.
\end{align*}
Using the Ricci identity (\ref{ricci}), equation (\ref{formula-psi-k}) is
easily obtained.
\end{proof}

A choice of an appropriate coordinate system simplifies substantially
the computation of the components $a_{i}^j$ of the Weingarten operator.
We describe how to obtain such a coordinate system. Fixed a point 
$x\in M,$ choose a geodesic coordinate system $(x^i)$ of $M$ around 
$x$ such that the coordinate vectors $\{Y_*\cdot \frac{\partial}{\partial x^i}|_x\}_{i=1}^n$
form a basis of principal directions of $\Sigma$ at $Y(x)$ and 
$\{\frac{\partial}{\partial x^i}|_x\}_{i=1}^n$ is an orthonormal basis with respect to 
the inner product given by the matrix $g=I+\nabla u\otimes\nabla u.$
Hence,
\[
a_i^j(x)=a_{ij}(x)=\frac{1}{W}u_{i;j}(x)\delta_{ij}=\kappa_i\delta_i^j 
\]
and
\[
G^{ij}=\frac{1}{W}F_k^ig^{kj}=\frac{1}{W}f_i\delta_k^i\delta^{kj}=\frac{1}{W} f_i\delta_i^j
\]
since $(F_i^j)$ is diagonal whenever $(a_i^j)$ is diagonal and
$g^{ij}=\delta^{ij}$. From now on we refer to such a
coordinate system as the  {\it special} coordinate system centered at $x.$ 

At the center of a {\it special} coordinate system the formula (\ref{formula-psi-k}) 
takes the simpler form
\begin{align}
\begin{split}
\label{formula-psi-k-2}
\sum_i f_iu_{k;ii}&= 2W\sum_i f_i\kappa_i u_iu_{i;k}+\frac{1}{W}\sum_jf_j\kappa_j
u^iu_{k;i} \\& -\sum_if_iR_{ilik}u^l+W(\Psi_{k}+\Psi_{t}u_k).
\end{split}
\end{align}


\section{The Height and Boundary Gradient Estimates}
\label{section3}

In this section we start establishing the {\it a priori} estimates of admissible solutions of the
Dirichlet problem (\ref{PD-EQ}). First
we consider the Theorem \ref{teorema2}. 
In this case the height estimate 
for admissible solutions is a direct consequence
of the existence of a  subsolution $\underline u$ satisfying the boundary
condition and of the inequality (\ref{gamma-longe-zero}).
In fact, it follows from the comparison principle applied to equation
(\ref{PD-EQ}) that $\underline{u}\leq u,$ which yields a lower bound.
An upper bound is obtained using as barrier the solution $\bar u$ 
of the Dirichlet problem
\begin{align}
\label{PD-mean-curvature}
\begin{split}
Q[\bar{u}] & =0  \hspace*{.5cm} \textrm{in}\,\, \Omega
\\ \bar u  & =\varphi \hspace*{.5cm} \textrm{on}\,\, \partial\Omega,
\end{split}
\end{align}
where $Q$ is the mean curvature operator. 
The assumption on the geometry of $\partial\Omega$ 
ensures the existence of such a solution $\overline{u}$ (see Theorem 1.5 in \cite{SPRUCK}).
So, it follows from the comparison principle 
for quasilinear elliptic equations that $u\leq\bar{u}.$  On the other hand,
since $\underline u= u= \bar{u}$ on $\partial \Omega,$ 
the inequality $\underline u\leq u\leq \bar{u}$
implies the boundary gradient estimate
\[
|\nabla u|<C  \hspace*{.5cm}\textrm{on} \,\, \partial \Omega.
\]
Hence the height and the boundary
gradient estimates are established in the case of Theorem \ref{teorema2}.
Now we consider the Theorem \ref{teorema1}. 
First note that the assumption on the existence of a bounded subsolution and the
solvability of (\ref{PD-mean-curvature}) ensures the height estimates.

The boundary gradient estimate is obtained following closely the ideas presented in 
\cite{TRU1}, which make use of the hypotheses 
on the boundary geometry to construct a lower barrier function.
Indeed let $d$ be the distance function to the boundary $\partial\Omega.$ 
In a small tubular neighborhood $\mathcal N=\{ x\in \Omega \, :\, d(x)<a\}$ of $\partial\Omega$
we define the barriers in the form
 \begin{equation}
 \label{lb-w}
 w=\varphi-f(d),
 \end{equation}
where  $f$ is a suitable real function and $a>0$ is a constant  chosen sufficiently
small to ensure that $d\in C^2(\bar{\mathcal{N}})$ (see \cite{LI-NIRENBERG}).
The boundary function $\varphi$ is redefined so that it is constant along normals to
$\partial\Omega$ in $\mathcal N$ and the function $f\in C^2 ([0,a])$ satisfies
$f'>0$ and $f''<0.$ Fixed a point $y_0$ in $\mathcal N,$ we fix around $y_0$
Fermi coordinates $(y^i)$ 
in $M$ along $\mathcal{N}_{d}=\{ x\in \Omega \, :\, d(x)=d(y_0)\},$  such that  
$y^n$ is the normal coordinate and the tangent 
coordinate vectors $\{\frac{\partial}{\partial y^\alpha}|_{y_{0}}\},$  $1\leq\alpha\leq n-1,$ 
form an orthonormal basis of eigenvectors that diagonalize 
$\nabla^2 d$ at $y_0.$ Since  
$\nabla d=\nu$ is the unit normal outward vector along $\mathcal{N}_{d}$ we have
\[
-\nabla^2 d(y_0)=\textrm{diag}(\kappa''_1, \kappa''_2, \ldots, \kappa''_{n-1}, 0),
\]
where $\kappa''=(\kappa''_1, \kappa''_2, \ldots, \kappa''_{n-1})$ denotes the vector of principal 
curvatures of $\mathcal{N}_{d}$ at $y_0.$ At $y_0$ we have
$w_i=\varphi_i$ for $i<n$. Moreover $w_n(y_0)=\varphi_n-f'$ and 
\begin{align}
\label{derivada-2-w}
\nabla^2 w=\nabla^2\varphi+ \textrm{diag}\big(f'\kappa'', -f''\big),
\end{align}
since $d_n=1$ and $d_i=0,$  $i<n.$ Therefore, the matrix of the Weingarten operator
of the graph of $w$ at $\big(y_0, w(y_0)\big)$ is
\begin{align*}
A[w]=&\, \big(g^{ik}(w)a_{jk}(w)\big)= \frac{1}{\sqrt{1+|\nabla w|^2}}\left( \delta^{jk}-\frac{w^jw^k}
{1+|\nabla w|^2}\right)w_{k;j}\\ =&\, O\Big(\frac{1}{v}\Big) + \tilde{A}[w]
\end{align*}
as $v\rightarrow\infty$ (or equivalently $f'\rightarrow\infty$) where $v=\sqrt{1+|\nabla w|^2}$ 
and we have written
\[
\tilde{A}[w]=\frac{1}{v}g^{-1}(w)\textrm{diag}(f'\kappa'', -f'').
\] 
It is convenient to split the computation of the matrix $\tilde{A}[w]$ into some blocks: For $i,j\leq n-1$
the components $\tilde{a}_{ij}$ of  $\tilde{A}[w]$ are
\[
\tilde{a}_{ij}=\frac{1}{v}\left( \delta^{jk}-\frac{\varphi^j\varphi^k}{1+v^2}\right)f'\kappa''_i\delta_{ik}
= \kappa''_i\delta_{ij}+O\Big(\frac{1}{v}\Big), \hspace*{.5cm} \textrm{as}\,\, v\rightarrow\infty.
\]
For $i=n$ and $j<n$ we have
\[
\tilde{a}_{nj}=-\frac{(f')^2\kappa''_j \varphi_j}{v^3}
= O\Big(\frac{1}{v}\Big),\hspace*{.5cm} \textrm{as}\,\, v\rightarrow\infty,
\]
and finally
\[
\tilde{a}_{nn}=-f''\frac{1+|\nabla\varphi|^2}{v^3}.
\]
Now we take $f$ of the form
\[
f(d)=\frac{1}{b}\log(1+cd)
\]
for positive constants $b,c>0$ to be determined. We have
\begin{align}
\label{conta-boundari-gradi-dp}
\begin{split}
f'(d)&=\frac{c}{b(1+cd)}\geq \frac{1}{b(1+ca)}
\\ f''(d)&=-bf'(d)^2,
\end{split}
\end{align}
so
\begin{equation}
\label{a-tilde-nn}
\tilde{a}_{nn}=b(f')^2\frac{1+|\nabla\varphi|^2}{v^3}=
\frac{b}{v}\left(1+O\Big(\frac{1}{v}\Big)\right), \hspace*{.5cm} \textrm{as}\,\, v\rightarrow\infty.
\end{equation}
Hence the principal curvatures $\tilde\kappa=(\tilde{\kappa}_1, \ldots, \tilde{\kappa}_n)$
of the graph of $w$ at $\big(y_0, w(y_0)\big)$ will differ from 
$\kappa''_1, \ldots ,\kappa''_{n-1}, \tilde{a}_{nn}$ by $O\Big(\frac{1}{v}\Big)$ as $v\rightarrow \infty.$
Then it follows from (\ref{a-tilde-nn}) that we may estimate
\[
\tilde{\kappa}_n\geq \frac{b}{2v}
\]
provided $v\geq v_0,$ $b\geq b_0,$ where $b_0$ and $v_0$ are 
constants depending on $|\varphi|_2$ and $\partial\Omega.$ Therefore
\begin{align}
\label{diferenca-kappas}
|\tilde{\kappa}_i-\kappa''_i|\leq \frac{b_1}{b}\tilde{\kappa}_n,
\end{align}
for a futher constant $b_1.$ On the other hand,
if $\tilde{y}_0=\tilde{y}_0(y_0)\in \partial\Omega$ denotes the closest point 
in $\partial\Omega$ to $y_0$ we thus estimate
\begin{align*}
\Psi(y_0, w)& \leq \Psi(\tilde{y}_0, \varphi)+|\Psi|_1 d
\\ & \leq \Psi(\tilde{y}_0, \varphi)+\frac{|\Psi|_1}{b v}
\\ &\leq f(\kappa', 0)+\frac{|\Psi|_1}{b v},
\end{align*}
where we used  (\ref{conta-boundari-gradi-dp}), the Serrin condition (\ref{SERRIN_CONDITION})
and the assumption $\Psi_t\geq 0.$
We recall that $\kappa'$ denotes the principal curvatures of $\partial\Omega.$
For $a>0$ small, we can replace $\kappa''_i$ by $\kappa'_i$ in (\ref{diferenca-kappas}).
On the other hand, using the mean value theorem and conditions (\ref{elipticidade})
and (\ref{f-CONVEXIDADE}) we obtain positive constants $\delta_0,\, t_0$ such that
\begin{align}
\label{seerin-condition-conta}
f(\tilde{\kappa})-f(\kappa', 0) \geq \delta_0t\tilde{\kappa}_n
\end{align}
whenever $t\leq t_0, \, |\tilde{\kappa}_i-\kappa'_i|\leq t\tilde{\kappa}_n, \, 
i=1, \ldots, n-1.$ To apply (\ref{seerin-condition-conta}) we should
observe that (\ref{elipticidade}) and (\ref{SERRIN_CONDITION}) imply $\tilde{\kappa}\in\Gamma.$
Then, to deduce the inequality $F[w]\geq \Psi$ as desired we fix $b$ so that
\[
b\geq b_0, \frac{b_1}{t_0}\quad\textrm{and}
\quad b^2\geq \frac{|\Psi|_1}{\delta_0t_0}.
\]
Setting $M=\sup (\varphi -u)$ we then choose $c$ and $a$ in such a way that
\[
ca=e^{b M}-1\quad\textrm{and}
\quad c\geq v_0b e^{b M}
\]
to ensure $v\geq v_0, \, w\leq u$ on $\partial \mathcal{N}.$ Therefore,
we find that $w$ is a lower barrier, that is,
\begin{align*}
F[w]=f(\tilde{\kappa}[w])&>\Psi \quad\textrm{in}\quad \mathcal{N}
 \\ w&\leq u \quad\textrm{on}\quad  \partial\mathcal{N},
\end{align*}
which implies $u\geq w$ in $\mathcal N.$ Since the condition (\ref{f-CONVEXIDADE}) 
implies that the mean curvature of $\partial\Omega$ is nonnegative,  we can conclude that 
there exists a solution  $\bar u$ of (\ref{PD-mean-curvature}) which is an
upper barrier. This establishes the boundary gradient estimates
in the Theorem \ref{teorema1}.  

\begin{remark}
\label{remark-1}
In the Lemma \ref{DEF-V} below we use again the function $w$ defined in (\ref{lb-w}).
We note that if (\ref{f-CONVEXIDADE}) holds, then for any $p\in\mathbb{R}^n,$  
we can choose the function $w$ as above in
such a way that it satisfies $\nabla^2 w\in \Gamma(p)$ on $\mathcal N.$ 
To see this first note that the matrix $g^{-1}(p)$ has eingenvalues
$1/(1+|p|^2)^{3/2}$ and $1/\sqrt{1+|p|^2}$ with multiplicities $1$ and $n-1,$ respectively.
Then the eigenvalues of the matrix $g^{-1}(p)\nabla^2 f(d)$ are 
\[
\tilde{\kappa}=\frac{1}{\sqrt{1+|p|^2}}\left(\tilde{\kappa}_1, \ldots, 
\frac{\tilde{\kappa}_n}{1+|p|^2}\right),
\]
where $\tilde{\kappa}'=(\tilde{\kappa}_1, \ldots, \tilde{\kappa}_n)$ are the eingenvalues
of $-\nabla^2 f(d).$ On the other hand, choosing $f$ as above we conclude  
$\tilde{\kappa}'=f'(\kappa'_1, \ldots , \kappa'_{n-1}, b f')$ on $\partial\Omega,$
where  $(\kappa'_1, \ldots,\kappa'_{n-1})$ denote the principal curvatures of $\partial\Omega.$
Hence, as the matrix $A(p,\nabla^2 w)$ has the form
\[
A(p,\nabla^2 w)=g^{-1}(p)\nabla^2 w=g^{-1}(p)\nabla^2 \varphi-g^{-1}(p)\nabla^2f(d),
\]
it follows from (\ref{f-CONVEXIDADE}) that for $f'$ sufficiently large (depending on $|p|,$ 
$|\varphi|_2$ and $\partial\Omega$) we have $\nabla^2w\in\Gamma(p)$ on $\partial\Omega.$
Since $\Gamma$ is open, the same holds in a small tubular neighborhood $\mathcal N$ of 
$\partial\Omega.$ It follows also from (\ref{derivada-2-w}) and  (\ref{conta-boundari-gradi-dp})
that the eigenvalues of $\nabla^2 w$ belongs to $\Gamma$ for $f'$ sufficiently large.
\end{remark}


\section{A Priori Gradient Estimates}
\label{section4}

In this section we derive (the interior) {\it a priori} gradient estimates for admissible solutions $u$
of the Dirichlet problem (\ref{PD-EQ}).
\begin{proposition}
\label{estimativa-gradiente} 
Let $u\in C^3(\Omega)\cap C^1(\overline{\Omega})$ be an admissible solution of (\ref{PD-EQ}).
Then, under the conditions (\ref{elipticidade})-(\ref{ki-negativo}),
\begin{equation}
\label{eq-estimativa-gradiente}
|\nabla u|\leq C \hspace*{.5cm}\textrm{in}\,\,\overline{\Omega},
\end{equation}
where $C$ depends on $|u|_0,$ $|\underline{u}|_1$ and other known data.
\end{proposition}
\begin{proof}
Set $\zeta(u) = ve^{Au}$, where $v=|\nabla u|^2=u^ku_k$ and 
$A$ is a positive constant to be chosen later. 
Let $x_0$ be a point where $\zeta$ attains its maximum. If $\zeta(x_0)=0$
then $|\nabla u|= 0$ and so the result is trivial. If $\zeta$ achieves its
maximum on $\partial\Omega$, then from the boundary gradient estimate obtained in
the last section (\ref{eq-estimativa-gradiente}) holds and we are done.
Hence, we are going to assume $x_0\in\Omega$ and $\zeta(x_0)>0$. 

We fix a normal coordinate system $(x^i)$
of $M$ centered at $x_0,$ such that
\[
\frac{\partial}{\partial x^1}\Big|_{x_0}=\frac{1}{|\nabla u|(x_0)}\nabla u(x_0).
\]
In terms of these coordinates we have $u_1(x_0)=|\nabla u(x_0)|>0$ and $u_j(x_0)=0$ for $j>1.$ 
Since $x_0$ is a maximum for $ \zeta,$
\begin{align*}
0=\zeta_i(x_0) = 2e^{2Au(x_0)}\left(Avu_i(x_0)+u^lu_{l;i}(x_0)\right)
\end{align*}
and the matrix $\nabla^2\zeta(x_0)=\{\zeta_{i;j}(x_0)\}$ is nonpositive. It follows that
\begin{equation}
\label{gradiente-conta1}
u^l(x_0)u_{l;i}(x_0)=-Av(x_0)u_i(x_0)
\end{equation}
for every $1\leq i\leq n.$ From now on all computations will be made at the point $x_0.$
As the matrix $\{G^{i;j}\}$ is positive definite one has
\[
 G^{ij}\zeta_{i;j}\leq 0.
\]
We compute
\begin{align*}
\zeta_{i;j} = &2e^{2Au}\left( u^lu_{l;ij}+u^l_{;i}u_{l;j}
+Avu_{i;j}+2Au^lu_{l;j}u_i\right.\\&\left.+2Au^lu_{l;j}u_i+2A^2vu_iu_j\right).
\end{align*}
Hence
\begin{align*}\begin{scriptsize}\begin{footnotesize}\begin{small}\end{small}\end{footnotesize}\end{scriptsize}
0\geq \frac{1}{2e^{2Au}}G^{ij}\zeta_{i;j}=& G^{ij}u^lu_{l;ij}+G^{ij}u^l_{;i}u_{l;j}+AvG^{ij}u_{i;j}
\\ &+4AG^{ij}u^lu_{l;j}u_j+2A^2vG^{ij}u_iu_j.
\end{align*}
It follows from (\ref{gradiente-conta1}) that
\[
4AG^{ij}u^lu_{l;i}u_j=-4A^2vG^{ij}u_iu_j
\]
and then
\begin{align}
\label{gradiente-conta2}
G^{ij}u^lu_{l;ij}+G^{ij}u^l_{;i}u_{l;j}-2A^2vG^{ij}u_iu_j+AvG^{ij}u_{i;j}\leq 0.
\end{align}
We use the formula (\ref{formula-psi-k}) at the Lemma \ref{lema-formula-psi-k} to obtain
\begin{align*}
G^{ij}u^lu_{l;ij}=& WG^{ij}a_j^ku^lu_{l;i}u_k+WG^{ij}a_j^ku^lu_{l;k}u_i+\frac{1}{W}G^{ij}a_{ij}u^lu^ku_{l;k}
\\ &-G^{ij}R_{iljk}u^lu^k+u^l\Psi_l+\Psi_t v.
\end{align*}
Since
\begin{align*}
R_{ijlk}u^lu^k=0 
\end{align*}
and
\begin{align*}
WG^{ij}a_j^ku^lu_{l;i}u_k= WG^{ij}a_j^k(-Avu_i)u_k=-AvWG^{ij}a_j^ku_iu_k
\end{align*}
and
\begin{align*}
 \frac{1}{W}G^{ij}a_{ij}u^lu^ku_{l;k}=\frac{1}{W}G^{ij}a_{ij}u^k(-Avu_k)=-\frac{1}{W}Av^2G^{ij}a_{ij}
\end{align*}
we get
\[
G^{ij}u^lu_{l;ij}=-2AvWG^{ij}a_j^ku_iu_k-\frac{Av^2}{W}G^{ij}a_{ij}+u^l\Psi_l+\Psi_tv.
\]
Plugging this expression back into (\ref{gradiente-conta2}) we get
\begin{align*}
&-2AvWG^{ij}a_j^ku_iu_k-\frac{Av^2}{W}G^{ij}a_{ij}+u^l\Psi_l+\Psi_tv \\
&+G^{ij}u^l_{;i}u_{l;j}-2A^2vG^{ij}u_iu_j+AvG^{ij}u_{i;j}\leq 0.
\end{align*}
Since $Wa_{ij}=u_{i;j}$ we can rewrite the above inequality as
\begin{align*}
 &G^{ij}u^l_{;i}u_{l;j}-2AWvG^{ij}a_j^ku_iu_k-2A^2vG^{ij}u_iu_j\\
&+\left(AvW-\frac{Av^2}{W}\right) G^{ij}a_{ij}+u^l\Psi_l+\Psi_tv\leq 0.
\end{align*}
Using the hypothesis $\Psi_t\geq 0$ we obtain
\begin{align}
\label{gradiente-conta3}
 G^{ij}u^l_{;i}u_{l;j}-2AWvG^{ij}a_j^ku_iu_k-2A^2vG^{ij}u_iu_j+\frac{Av}{W}G^{ij}a_{ij}
 +\Psi_lu^l\leq 0.
\end{align}
From the choice of the coordinate system
and (\ref{gradiente-conta1}) it follows that
\[
u_{1;1}=-Av \hspace*{.5cm} \textrm{and} \hspace*{.5cm} u_{1;i}=u_{i;1}=0 \hspace*{.5cm}(i>1).
\]
After a rotation of the coordinates $(x^2, \ldots , x^n)$ we may assume  that
$\nabla^2 u=\{u_{i;j}(x_0)\}$ is diagonal. Since
\[
a_i^j=g^{jk}a_{ki}=\frac{1}{W}\left( \sigma^{jk}-\frac{u^ju^k}{W^2}\right)u_{k;i},
\]
at $x_0,$ we  have
\begin{align*}
&a_i^j=0 \hspace*{.5cm} (i\neq j) \\ &a_1^1=\frac{1}{W^3}u_{1;1}=-\frac{Av}{W^3}<0
\\& a_i^i=\frac{1}{W}u_{i;i} \hspace*{.5cm} (i>1).
\end{align*}
From Lemma \ref{derivada-general-f}, the matrix $\{F_i^j\}$ is diagonal. Then
$\{G^{ij}\}$ is also diagonal with
\begin{align*}
 G^{ii}=\frac{1}{W}F_k^ig^{ki}=\frac{1}{W}f_i \hspace*{.5cm} \textrm{and} \hspace*{.5cm} 
  G^{11}=\frac{1}{W}F_k^1g^{k1}=\frac{1}{W^3}f_1.
\end{align*}
Using these relations and discarding the term
\[
\frac{Av}{W}G^{ij}a_{ij}=\frac{Av}{W^2}\sum_if_i\kappa_i \geq 0
\]
we get from (\ref{gradiente-conta3}) the inequality
\begin{align*}
G^{ii}u_{i;i}^2-2AWvG^{11}a_1^1(u_1)^2-2A^2vG^{11}(u_1)^2+\Psi_1u_1\leq 0,
\end{align*}
which may be rewritten as
\begin{align*}
\sum_{\alpha>1}G^{\alpha\alpha}u_{\alpha;\alpha}^2+G^{11}
\left(\frac{2A^2v^3}{W^2}-2A^2v^2+A^2v^2\right)+\Psi_1\sqrt{v}\leq 0 .
\end{align*}
Since
\[
\frac{2A^2v^3}{W^2}-2A^2v^2+A^2v^2=\frac{A^2v^3-A^2v^2}{(1+v)^2}
\]
we have
\begin{align*}
\sum_{\alpha>1}G^{\alpha\alpha}u_{\alpha;\alpha}^2
+\frac{A^2v^3-A^2v^2}{(1+v)^2}G^{11}+\Psi_1\sqrt{v} \leq 0.
\end{align*}
Then
\[
\frac{A^2v^3-A^2v^2}{(1+v)^2} \frac{1}{W^3}f_1\leq -\Psi_1\sqrt{v}\leq |D\Psi| \sqrt{v}.
\]
Once
\[
\kappa_1=a_1^1=-\frac{Av}{W^3}<0,
\]
we may apply hypothesis (\ref{ki-negativo}) to get $f_1\geq c_0>0,$ which implies 
\[
\frac{A^2v^3-A^2v^2}{W^5\sqrt{v}}\leq \frac{|D\Psi|}{c_0} .
\]
Now we choose
\[
A=\left( \frac{2}{c_0}\sup_{M\times I} |D\Psi|\right)^{1/2},
\]
where $I$ is the interval $I=[-C, C]$ with $C$ being a uniform constant that satisfies $|u|_0<C.$ 
Therefore,
\[
\frac{(u_1)^3((u_1)^2-1)}{(1+(u_1)^2)^{5/2}}\leq \frac{1}{2},
\]
that is,
\[
(u_1)^5-(u_1)^3-\frac{1}{2}(1+(u_1)^2)^{5/2}<0.
\]
Since $u_1>0$ this yields a bound for $u_1$ and hence for $\zeta(x_0),$ which implies the desired estimate. 
\end{proof}


\section{Boundary Estimates for Second Derivatives}
\label{section5}

In this section we establish the crucial  {\it a priori} second derivatives estimates at 
the boundary.
Bounds for pure tangential derivatives follow from the relation 
$u=\varphi$ on $\partial\Omega.$ It remains to estimate the mixed and double normal derivatives. 

Consider the linearized operator 
\[
L=G^{ij}-b^i,
\]
where $b^i=\frac{1}{W^2}\sum_j f_j\kappa_j u^i.$
It follows from  (\ref{relacao_euler}), (\ref{upper-f})
and (\ref{eq-estimativa-gradiente}) that $|b^i|\leq C$ for a uniform constant $C.$

To proceed, we first derive some key preliminary lemmas. Let $x_0$ be a point on 
$\partial\Omega.$ Let $\rho(x)$ denote the distance from $x$ to $x_0,$ 
$\rho(x)=\textrm{dist}(x,x_0),$ and set
\[
\Omega_{\delta}=\{x\in\Omega\, :\, \rho(x)<\delta\}.
\]
Since $(\rho^2)_{\,i;j}(x_0)= 2\sigma_{ij}(x_0)$, by choosing $\delta>0$ sufficiently small we 
may assume $\rho$ smooth in $\Omega_{\delta}$ and
\begin{equation}
\label{RHO-2}
\sigma_{ij}\leq (\rho^2)_{\,i;j}\leq 3\sigma_{ij} \hspace*{.5cm}\textrm{in}\,\, \Omega_\delta.
\end{equation}
Since $\partial\Omega$ is smooth,  we may also assume that the distance function $d(x)$ to the 
boundary $\partial\Omega$ is smooth in $\Omega_\delta.$ In what follows, we redefine
the boundary function $\varphi$ in $\Omega_\delta$ as being  
constant along the normals to $\partial \Omega.$ 
 
Let $\xi$ be a $C^2$ arbitrary vector field defined in $\Omega_\delta$ and  $\eta$ any extension 
to $\Omega_\delta$ of the restriction of the vector field $\nabla u$ to $\partial\Omega\cap\partial
\Omega_{\delta}.$ For instance, we can take $\eta$ as the parallel transport of $\nabla u$ along the
normals geodesics to $\partial\Omega$ in $\Omega_\delta.$
Inspired in the approach used by Ivochkina in  \cite{IVO2} we define the function
\begin{equation}
\label{CARA-W}
w=\langle\nabla u , \xi\rangle- \langle\nabla \varphi , \xi\rangle-\frac{1}{2}|\nabla u-\eta|^2.
\end{equation}
The function $w$ satisfies a fundamental inequality.
\begin{proposition}
\label{DEF-W}
Assume that $f$ satisfies (\ref{elipticidade})-(\ref{relacao_euler}). Then the function $w$ satisfies
\begin{equation}
\label{LW}
L[w]\leq C(1+ G^{ij}\sigma_{ij}+G^{ij}w_iw_j) \quad \textrm{in}  \quad \Omega_\delta,
\end{equation}
where $C$ is a uniform positive constant.
\end{proposition}
\begin{proof}
For convenience we denote $\mu = \langle\nabla \varphi , \xi\rangle.$
First we calculate the derivatives of $w$ in an arbitrary coordinate system.
We have
\begin{align*}
w_i &=  \langle  \nabla_i \nabla u , \xi \rangle+\langle\nabla u , \nabla_i\xi\rangle
-\mu_i-\langle \nabla_i\nabla u-\nabla_i\eta , \nabla u-\eta \rangle \\
&=\left(\xi^k +\eta^k- u^k\right)u_{k;i} +\left((\xi^k)_i +(\eta^k)_i\right)u_k
-\mu_i- \langle \nabla_i\eta ,\eta\rangle
\end{align*}
and
\begin{align*}
w_{i;j} =& \langle\nabla_j\nabla_i\nabla u, \xi\rangle+\langle\nabla_i\nabla u , \nabla_j \xi\rangle
+\langle\nabla_j\nabla u , \nabla_i\xi\rangle +\langle\nabla u , \nabla_j\nabla_i\xi\rangle \\ & 
-\mu_{i;j}-\langle\nabla_j\nabla_i\nabla u -\nabla_j\nabla_i\eta , \nabla u-\eta\rangle
-\langle\nabla_i\nabla u -\nabla_i\eta , \nabla_j\nabla u-\nabla_j\eta\rangle
\\=& \left(\xi^k +\eta^k-u^k\right)u_{k;ij}+\left( (\xi^k)_j+(\eta^k)_j\right)u_{k;i} 
+\left( (\xi^k)_i+(\eta^k)_i\right)u_{k;j}  \\ & -u^k_{;i}u_{k;j}
+\left( (\xi^k)_{i;j}+(\eta^k)_{i;j}\right)u_k
-\mu_{i;j} -\langle \nabla_i\eta ,\nabla_j\eta\rangle -\langle \nabla_j\nabla_{i}\eta ,\eta\rangle,
\end{align*}
where we denote by $\xi^k, (\xi^k) _i$ and $(\xi^k) _{i;j}$ the components of the vectors
$\xi, \nabla_i\xi$ and $\nabla_j\nabla_i\xi,$ respectively (the same notation is used for $\eta$).
Therefore,
\begin{align*}
G^{ij}w_{i;j} =& \left(\xi^k +\eta^k-u^k\right)G^{ij}u_{k;ij}+2G^{ij}\left( (\xi^k)_j+(\eta^k)_j\right)u_{k;i} 
  -G^{ij}u^k_{;i}u_{k;j}\\ & 
+G^{ij}\Big(\left( (\xi^k)_{i;j}+(\eta^k)_{i;j}\right)u_k
-\mu_{i;j} -\langle \nabla_i\eta ,\nabla_j\eta\rangle -\langle \nabla_j\nabla_{i}\eta ,\eta\rangle
\Big).
\end{align*}
Now we use (\ref{formula-psi-k}) to obtain 
\begin{align*}
\left(\xi^k +\eta^k-u^k\right)  G^{ij} u_{k;ij} = &W\left(\xi^k+\eta^k-u^k\right)G^{ij}a_j^lu_{k;i}u_l
+W\left(\xi^k+\eta^k-u^k\right)\\ &\times G^{ij}a_j^lu_{k;l}u_i  
+\frac{1}{W}\left(\xi^k+\eta^k-u^k\right)G^{jl}a_{jl}u^iu_{k;i} \\
&+\left(\xi^k+\eta^k-u^k\right)\left(\Psi_k+\Psi_t u_k-G^{ij}R_{iljk}u^l\right).
\end{align*}
On the other hand, it follows from the expression for $w_i$ that
\begin{align*}
\left(\xi^k+\eta^k-u^k\right)u_{i;k} &= w_i-\left((\xi^k)_i+(\eta^k)_i\right)u_k
+\mu_i+\langle \nabla_i\eta ,\eta\rangle .
\end{align*}
Substituting these equalities in the above equation we get
\begin{align}
\begin{split}
\label{CONTA-HESSIAN-BOUNDARY-2}
G^{ij}&w_{i;j}=WG^{ij}a_j^lw_iu_l+WG^{ij}a_j^lw_lu_i+\frac{1}{W}G^{jl}a_{jl}u^iw_i
-G^{ij}u^k_{;i}u_{k;j} \\&+2G^{ij}\left( (\xi^k)_j+(\eta^k)_j\right)u_{k;i}+WG^{ij}a_j^lu_l 
\Big(\mu_i-\left((\xi^k)_i+(\eta^k)_i\right)u_k \\ &+\langle \nabla_i\eta ,\eta\rangle\Big)
 +WG^{ij}a_j^lu_i \Big(\mu_l-\left((\xi^k)_l+(\eta^k)_l\right)u_k
+\langle \nabla_l\eta ,\eta\rangle\Big)+\\&\frac{1}{W}G^{jl}a_{jl}u^i 
  \Big(\mu_i-\left((\xi^k)_i+(\eta^k)_i\right)u_k 
+\langle \nabla_i\eta ,\eta\rangle\Big)\\&
+G^{ij}\Big(\left( (\xi^k)_{i;j}+(\eta^k)_{i;j}\right)u_k
-\mu_{i;j} -\langle \nabla_i\eta ,\nabla_j\eta\rangle 
-\langle \nabla_j\nabla_{i}\eta ,\eta\rangle\\ &-\left(\xi^k+\eta^k-u^k\right)
R_{iljk}u^l\Big)+\left(\xi^k+\eta^k-u^k\right)(\Psi_k+\Psi_t u_k).
\end{split}
\end{align}
Since (\ref{LW}) does not depend on the coordinate system, i.e., it is a tensorial
inequality,  it is sufficient to prove it in a fixed  coordinate system. 
Given $x\in\Omega,$ let $(x^i)$ be the {\it special} coordinate system centered at $x.$ 
In terms of this coordinates the inequality (\ref{LW}) takes at $x$ the form
\begin{equation}
\label{LW-2}
L[w]=\frac{1}{W}\sum_i f_iw_{i;i}-b^iw_i\leq C\left(1+\frac{1}{W}\sum_i f_i\sigma_{ii}+
\frac{1}{W}\sum_i f_iw_i^2\right).
\end{equation}
We will prove the above inequality. In what follows all computations are done at the point $x$.

In these coordinates we have $\kappa_i =a_i^j=a_{ij}=\frac{1}{W}u_{i;j}\delta_{ij}$ and
$ G^{ij}=\frac{1}{W}f_i\delta_i^j.$ Since the quantities depending on 
$\nabla u, \, \xi, \, \eta$ and $\mu$ are under control, we get
\begin{align*}
WG^{ij}a_j^lw_iu_l&= \sum_i f_i\kappa_i w_iu_i\leq \varepsilon\sum_i f_i \kappa_i^2 
+\frac{1}{\varepsilon}\sum_i f_i w_i^2u_i^2 
\\ &\leq \varepsilon \sum_i f_i\kappa_i^2 + C \sum_i f_i w_i^2
\\ 2G^{ij}\left( (\xi^k)_j+(\eta^k)_j\right)u_{k;i}&= 
2\left( (\xi^i)_i+(\eta^i)_i\right)f_i\kappa_i\leq  
\varepsilon\sum_i f_i \kappa_i^2 +C\sum_i f_i\\
WG^{ij}a_j^lu_l \mu_i & = \sum_i f_i\kappa_i u_i\mu_i 
\leq \varepsilon\sum_i f_i \kappa_i^2 +C\sum_i f_i\\
G^{ij}u^k_{;i}u_{k;j}&= G^{ij}\sigma^{kl}u_{l;i}u_{k;j} = W^2G^{ij}\sigma^{kl}a_{l;i}a_{k;j} 
=W\sigma^{ii} f_i \kappa_i^2 \\ &\geq C_0 \sum_i f_i \kappa_i^2 \\
\frac{1}{W}G^{jl}a_{jl}u^i \mu_i &=\frac{1}{W^2} \sum_j f_j\kappa_j u^i\mu_i\leq C\sum_j f_j\kappa_j \leq C
\\ G^{i;j}\mu_{i;j}&=\frac{1}{W}\sum_i f_i\mu_{i;i}\leq C\sum_i f_i,
\end{align*}
where $\varepsilon>0$ is any positive number and $C_0>0$ depends
only on $\sigma|_{\overline\Omega}.$ To obtain the above 
inequalities we made use of the ellipticity condition
$f_i>0.$ Estimating all the terms in 
(\ref{CONTA-HESSIAN-BOUNDARY-2}) as above, we 
conclude that  equality 
(\ref{CONTA-HESSIAN-BOUNDARY-2}) implies 
\begin{align}
\label{CONTALW1}
\begin{split}
L[w]\leq (\varepsilon C-C_0)\sum_i f_i\kappa_i^2+C\sum_i f_i w_i^2+C\sum_i f_i+C.
\end{split}
\end{align}
Choosing $\varepsilon>0$ sufficiently small such that the first term on the sum above
becomes negative we obtain 
\[
L[w] \leq C\big(1+\sum_i f_i+\sum_i f_i w_i^2\big).
\]
Using that $\sigma_{ii}>C_0>0 \, $ in $\overline\Omega$ and $W$ is under control,
we get (\ref{LW-2}).
\end{proof}
We note that inequality (\ref{LW}) may be simplified further. In fact, since 
\[
G^{ij}\sigma_{ij}\geq \delta_0>0,
\]
replacing $C$ to $C/\delta_0+C$ (we may assume $1>\delta_0>0$) we get
\begin{equation}
\label{LW3}
L[w]\leq C(G^{ij}\sigma_{ij}+G^{ij}w_iw_j) \hspace*{.5cm} \textrm{in}  \,\, \Omega_\delta.
\end{equation}

Setting 
\begin{equation}
\label{w-til}
\tilde{w}= 1-e^{-a_0 w}
\end{equation}
for a positive constant  $a_0,$ we get $\tilde{w}_i = a_0e^{-a_0w}w_i $ and
$\tilde{w}_{i;j} = a_0e^{-a_0w}\left(w_{i;j}-a_0w_iw_j\right).$ Therefore,
\begin{align*}
L[\tilde{w}] &= G^{ij}\tilde w_{ij}-b^i\tilde w_i= a_0e^{-a_0w}(L[w]-a_0 G^{ij}w_iw_j),
\end{align*}
if we choose $a_0$ large such that $a_0\geq C,$ where $C$ is the constant in (\ref{LW3}),
\[
L[w]-a_0 G^{ij}w_iw_j\leq L[w]-C G^{ij}w_iw_j\leq CG^{ij}\sigma_{ij}.
\]
Hence
\begin{align}
\label{LW-TIL}
L[\tilde{w}] \leq   CG^{ij}\sigma_{ij}.
\end{align}
Now we extend to curvature equations the Lemma 6.2 in \cite{GUAN-2} 
obtained by Guan to Hessian equations. This lemma gives the elements to 
complete the construction of a barrier function.
\begin{lemma}
\label{DEF-V}
Assume that $f$ satisfies  (\ref{elipticidade})-(\ref{produto}). 
Then there exist some uniform posi\-tive constants $t, \delta , \varepsilon$ sufficiently small and 
$N$ sufficiently large such that the function
\begin{equation}
\label{V}
v=(u-\underline{u}) +td-\frac{N}{2}d^2
\end{equation}
satisfies
\begin{equation}
\label{LV<0}
L[v]\leq -\varepsilon(1+G^{ij}\sigma_{ij}) \quad \textrm{in}  \quad \Omega_\delta
\end{equation}
and
\[
v\geq 0 \quad\textrm{on} \quad \partial\Omega_\delta.
\]
\end{lemma}
\begin{proof}
Since $\underline u$ is locally strictly convex in a neighborhood of 
$\partial\Omega$ we may choose $\delta>0$ small enough such that the eigenvalues 
$\lambda(\nabla^2\underline u) \in \Gamma^+$ in $\Omega_\delta.$ 
In particular, we have $\nabla^2\underline u\in\Gamma(\nabla u)$ 
in  $\Omega_\delta.$ Consider the function $v^*=\underline{u}-3\varepsilon\rho^2$. 
Since $\Gamma(\nabla u)$ is open 
and $F[\underline u]>0,$ we may choose $\varepsilon>0$ sufficiently small, 
such that $v^*$ is admissible and $\nabla^2 v^*\in \Gamma(\nabla u)$
in $\Omega_\delta$. 

It follows from the concavity of $G(p, \cdot)$ that
\[
G^{ij}(p,r)(r_{ij}-s_{ij})\leq G(p,r)-G(p,s) \quad {\rm for \,all}\quad  r,s\in\Gamma(p).
\]
Applying this inequality we get
\begin{align*}
L[u-\underline{u}]&=L[u-v^*-3\varepsilon\rho^2]
\\ &=G^{ij}(u_{i;j}-v^*_{i;j})-b^i(u_i-v^*_i)-3\varepsilon L[\rho^2] 
\\ & \leq  G(\nabla u,\nabla^2u)-G(\nabla u, \nabla^2v^*)-b^i(u_i-
v^*_i)\\ &-3\varepsilon G^{ij}(\rho^2)_{\, i;j}+6\varepsilon\rho b^i\rho_i.
\end{align*}
Since $G(\nabla u,\nabla^2u)=\Psi$ and $G(\nabla u, \nabla^2v^*)>0,$
it follows from the $C^1$ estimate and the boundedness of $b^i$ that
\[
L[u-\underline{u}]\leq C_1-3\varepsilon G^{ij}(\rho^2)_{i;j}.
\]
Hence, we conclude from (\ref{RHO-2})
\begin{equation}
\label{CONTALV-1}
L[u-\underline{u}]\leq C_1-3\varepsilon G^{ij}\sigma_{i;j}.
\end{equation}
As in the previous lemma, the inequality proposed is a tensorial one.
So, it is sufficient to prove (\ref{LV<0}) in a fixed coordinate system.
For $\delta>0$ small we may define Fermi
coordinates $(y^i)$ on $\Omega_\delta$ along $\partial\Omega,$
such that $y^n=d$ is the normal coordinate.
In these coordinates we have $d_\alpha=0, \, 1\leq\alpha \leq n-1,$ and  $d_n=1.$
Hence, a straightforward computation yields
 \[
 L\left[td-\frac{N}{2}d^2\right]=(t -dN)L[d]-NG^{nn}.
 \]
Since there exists a uniform positive constant $C$ that satisfies 
$d_{i;j}\leq C \sigma_{ij} $
in $\Omega_\delta$ and $|b^i|<C,$ we have
 \[
L\left[td-\frac{N}{2}d^2\right]  \leq C_2(t+N\delta)(1+G^{ij}\sigma_{ij})-NG^{nn}.
 \]
This inequality and (\ref{CONTALV-1}) give
\begin{align*}
L[v] &\leq   L[u-\underline{u}]+L\left[td-\frac{N}{2}d^2\right] 
\\ &\leq  C_1-3\varepsilon G^{ij}\sigma_{ij}+
C_2(t+N\delta)(1+G^{ij}\sigma_{ij})-NG^{nn} 
\\ &= C_1+C_2(t+N\delta)+\left(C_2(t+N\delta)-
3\varepsilon\right)G^{ij}\sigma_{ij}-NG^{nn}.
\end{align*}
As in \cite{GUAN-2}, we choose indices such that 
$f_1\ge \cdots \geq f_n.$ Since the eigenvalues of the matrix $G^{ij}$ 
are $\frac{1}{W}f_1, \ldots , \frac{1}{W} f_n,$ it follows from our choice of indices that 
\[
G^{nn}\geq \frac{1}{W}f_n\geq Cf_n \hspace*{.5cm} \textrm{and} \hspace*{.5cm}
G^{ij}\sigma_{ij}\geq C\sum_ if_i.
\]
Using the arithmetic-geometric mean 
inequality and  (\ref{produto})  we get
\begin{align*}
\varepsilon G^{ij}\sigma_{ij}+NG^{nn} & \geq 
C\sum_{i}f_i + CNf_n 
\\&\geq C n\varepsilon(Nf_1\cdot \ldots \cdot f_n)^{1/n}=C_3N^{1/n}.
\end{align*}
Now we apply this relation into the above inequality to get
 \[
 L[v] \leq C_1+C_2(t+N\delta)+ (C_2(t+N\delta)-2\varepsilon)G^{ij}\sigma_{ij}-C_3N^{1/n}.
 \]
Since  $\delta^2\leq t\delta/N$ implies $t\delta-N/2\delta^2\geq 0$ and $u\geq\underline u,$ 
we choose $t=\frac{\varepsilon}{2C_2}$ and $\delta\leq \frac{t}{N}$ to get 
$v\geq0$ on $\Omega \cap\partial\Omega_\delta$.  With this choice we  have
 \[
 L[v] \leq C_1-\varepsilon G^{ij}\sigma_{ij}-C_3N^{1/n}.
 \]
By choosing $N$ large such that  $C_3N^{1/n}\geq C_1+2\varepsilon$
we obtain (\ref{LV<0}).
\end{proof}
\begin{remark}
Under the hypotheses of Theorem \ref{teorema1} we 
construct a subsolution $w$ defined in $\Omega_\delta$ and 
that is not necessarily strictly convex but satisfies 
$\nabla^2w\in\Gamma(\nabla u).$ We replace $\underline u$ by $w$ in the Lemma
 above to get the result. See Remark \ref{remark-1}.
\end{remark}

\begin{remark}
\label{quotient-rem}
We claim that the quotients $S_{k,l}^{\frac{1}{k-l}}$ with 
$1\leq l < k\leq n$ satisfy (\ref{produto}) in $\Gamma_\Psi$. 
For proving that, we first observe that
\[
\frac{\partial}{\partial\lambda_i}  S_{k,l}^{\frac{1}{k-l}}= \frac{1}{k-l}S_{k,l}^{\frac{1}{k-l}-1}
\frac{\partial S_{k,l}}{\partial \lambda_i}\geq C_1 \frac{\partial S_{k,l}}{\partial \lambda_i}
\]
for some positive constant $C_1 = C_1(\Psi_0,\Psi_1)$. 
Therefore in order to prove the claim, it suffices to verify that
\begin{equation}
\label{ineq-aux}
\bigg(\frac{\partial S_{k,l}}{\partial \lambda_1}\cdot \ldots\cdot 
\frac{\partial S_{k,l}}{\partial \lambda_n}\bigg)^{\frac{1}{n}}\geq C_0,
\end{equation}
for some positive constant $C_0$.  Notice that
\begin{eqnarray*}
\frac{\partial S_{k,l}}{\partial \lambda_i} & = & 
\frac{S_l \frac{\partial S_{k}}{\partial \lambda_i} - 
S_k \frac{\partial S_{l}}{\partial \lambda_i} }{S_l^2}\\
& = & \frac{\big(\frac{\partial S_{l+1}}{\partial \lambda_i} +
\lambda_i \frac{\partial S_{l}}{\partial \lambda_i}\big) 
\frac{\partial S_{k}}{\partial \lambda_i} - \big(\frac{\partial S_{k+1}}{\partial \lambda_i} +
\lambda_i \frac{\partial S_{k}}{\partial \lambda_i}\big) \frac{\partial S_{l}}{\partial \lambda_i} }{S_l^2}\\
& = & \frac{\frac{\partial S_{l+1}}{\partial \lambda_i}  \frac{\partial S_{k}}{\partial \lambda_i} - 
\frac{\partial S_{k+1}}{\partial \lambda_i} \frac{\partial S_{l}}{\partial \lambda_i} }{S_l^2}\\
& \geq & \frac{n(k-l)}{k(n-l)}\frac{\frac{\partial S_{l+1}}{\partial\lambda_i}
\frac{\partial S_{k}}{\partial\lambda_i}}{S^2_l},
\end{eqnarray*}
where have used the generalized Newton-Maclaurin inequalities for the last step (for details, see
\cite{Spruck-fully},  p. 14 or \cite{TRU1}). Since 
\[
\left(\frac{1}{{n\choose k}} S_k(\lambda)\right)^{\frac{1}{k}}\leq
\left(\frac{1}{{n \choose l}}S_{l}(\lambda)\right)^{\frac{1}{l}}, 
\quad \lambda \in \Gamma_k\subset \Gamma_l,
\]
for $1\leq l < k\leq n,$ we conclude from  \eqref{elipticidade} and \eqref{gamma-longe-zero}  that
\begin{eqnarray*}
\frac{\partial S_{k,l}}{\partial \lambda_i} \geq  \frac{n(k-l)}{k(n-l)}\frac{\frac{\partial S_{l+1}}
{\partial\lambda_i}\frac{\partial S_{k}}{\partial\lambda_i}}{{n \choose l}\left(\frac{\delta}{n}\right)^{2l}}.
\end{eqnarray*}
Now, since that the cone for $f=S_{k,l}^{\frac{1}{k-l}}$ is 
$\Gamma_k$ and $\Gamma_l \subset \Gamma_k$ for $l<k$ we have 
that there exist positive constants $C_1$ and $C_2$ such that (see, for instance, \cite{GUAN-2} or \cite{YYLI})
\begin{equation*}
\bigg(\frac{\partial S_{l+1}}{\partial \lambda_1}
\cdot \ldots\cdot \frac{\partial S_{l+1}}{\partial \lambda_n}\bigg)^{\frac{1}{n}}\geq C_1
\end{equation*}
 and 
\begin{equation*}
\bigg(\frac{\partial S_{k}}{\partial \lambda_1}\cdot \ldots\cdot \frac{\partial S_{k}}{\partial
 \lambda_n}\bigg)^{\frac{1}{n}}\geq C_2
\end{equation*}
for all $\lambda\in \Gamma_\Psi$. We conclude that (\ref{ineq-aux}) 
holds for all $\lambda\in \Gamma_\Psi$. This proves the claim.
\end{remark}

\subsection{Mixed Second Derivative Boundary Estimate}
We define the function
\begin{equation}
\label{CARA-H}
h=\tilde{w}+b_0 \rho^2+c_0v,
\end{equation}
where $b_0$ and $c_0$ are constants to be chosen later. Assume  the
vector field $\xi$ is tangent along $\partial\Omega\cap\partial\Omega_\delta.$ 
Hence $\tilde{w}=0$ on $\partial\Omega\cap\partial\Omega_\delta.$ 
Then, since $v\geq 0$ on $\partial\Omega_\delta,$
if $\delta>0$ is small enough and $b_0$ is sufficiently large we have
$h\geq 0$ on  $\partial\Omega_\delta.$ On the other hand, it follows from (\ref{RHO-2}), 
(\ref{LW-TIL}) and (\ref{LV<0}) that
\begin{align*}
L[h]&=L[\tilde{w}]+b_0L[\rho^2]+c_0L[v] 
\\ & \leq (C+Cb_0-c_0\epsilon)(1+G^{ij}\sigma_{ij})+Cb_0.
\end{align*}
Therefore, for $c_0\gg b_0\gg 1$ both sufficiently large, we get $L[h] \leq 0$ in $\Omega_\delta$
and $h\geq 0$ on $\partial\Omega_\delta.$ It follows from the maximum principle that $h\geq 0$ in 
$\Omega_\delta.$ Consequently,
\[
\nabla_\nu h(x_0)\geq 0,
\]
which give us
\[
u_{\xi;\nu}(x_0) \geq -\langle \nabla u,
\nabla_\nu\xi\rangle(x_0)-\frac{c_0}{a_0}(u-\underline u)_\nu(x_0)-\frac{c_0}{a_0}t,
\]
where $\nu$ is the interior unit normal vector field to $\partial\Omega.$
Replacing $\xi$ by $-\xi$ at the definition of $w$ 
we establish a bound  for the 
mixed normal-tangential derivatives on $\partial\Omega$
\[
|u_{\xi;\nu }(x_0)|\leq C,
\]
for any direction tangent $\xi$ to $\partial\Omega.$ Since $x_0$ is arbitrary, we have
\begin{equation}
\label{TANGENTE-NORMAL}
|u_{\xi;\nu}|<C \hspace*{.5cm}\,\, \partial\Omega.
\end{equation}

\subsection{Double Normal Second Derivative Boundary Estimate}
For the pure normal second derivative,  since 
$\sum_i\kappa_i[u]\geq \delta_0>0,$
we need only to derive an upper bound
\begin{equation}
\label{NORMAL-NORMAL}
u_{\nu;\nu} \leq C \hspace*{.5cm} \textrm{on}\,\, \partial\Omega.
\end{equation}

First we  note that the equality 
$u=\varphi$ on $\partial \Omega$ implies 
\begin{align}
\label{RELACAO-U-PHI}
u_{\xi;\eta}(y)&=\varphi_{\xi;\eta}(y)-u_\nu(y)  \Pi(\xi, \eta)(y),
\end{align}
for any tangent vectors $\xi, \eta \in T_y(\partial\Omega)\subset T_yM,$ 
$y \in \partial\Omega,$ where
$\Pi$ denotes the second fundamental form of $\partial\Omega.$
Let $T_u$ be the $(0,2)$ tensor defined on $\partial\Omega$ by
\begin{equation}
\label{DEF-T}
T_u=\big(\tilde\nabla^2\varphi-u_\nu \Pi\big),
\end{equation}
where $\tilde\nabla$ is the induced connection on $\partial\Omega.$
Since $a_{\alpha\beta}=\frac{1}{W}u_{\alpha;\beta},$ it follows from
the equality (\ref{RELACAO-U-PHI}) that the
components of $T_u$ in terms of coordinates $(y^\alpha)$ in $\partial\Omega$ are $W  a_{\alpha\beta}.$
We denote by $\tilde\kappa=(\tilde\kappa_1, \ldots ,\tilde\kappa_{n-1})$
the eigenvalues of the tensor $T_u$ with respect to the inner product defined
on $\partial\Omega$ by the matrix $\tilde g=\tilde\sigma+\tilde{\nabla}\varphi\otimes
\tilde{\nabla}\varphi,$ where $\tilde\sigma$ is the metric on $\partial\Omega$ induced 
by $\sigma.$ 

Let $\Gamma'$ be the projection of $\Gamma$ on $\mathbb{R}^{n-1}.$ We denote by $d(\kappa')$ 
the distance from $\kappa'\in \Gamma'$ to $\partial\Gamma'.$
We point out that $\Gamma'$ is also an open convex symmetric cone. 

We are going to analyze the behavior of $d(\kappa'[u])$ for admissible solutions $u.$
First we fix Fermi coordinates $(y^i)$ 
in $M$ along $\partial\Omega,$  such that  
$y^n$ is the normal coordinate and the tangent 
coordinate vectors $\{\frac{\partial}{\partial y^\alpha}|_{y_0}\},$  $1\leq\alpha\leq n-1,$ 
is an orthonormal basis of eigenvectors that diagonalize 
$T_u$ at the given $y_0\in\partial\Omega,$ 
with respect to the inner product $\tilde g=\tilde\sigma+\tilde{\nabla}\varphi\otimes
\tilde{\nabla}\varphi$. At $y_0$ the matrix of the second fundamental form
of $\Sigma$ in terms of this coordinate system is given by
\begin{align}
\label{matrix-second-form}
 a_{ij}  = \frac{1}{W}u_{i;j}\delta_{ij}, \, (i,j<n), \hspace*{.5cm}
 a_{in} = \frac{1}{W}u_{\nu ;i}, \, (i<n), \hspace*{.5cm}
 a_{nn} = \frac{1}{W} u_{\nu;\nu}.
\end{align}

It follows from (\ref{RELACAO-U-PHI}) that $\tilde\kappa=( u_{1;1}, \ldots,  u_{n-1;n-1})$ 
are also the eigenvalues of the tensor $T_u$ defined above. 
Since the principal curvatures $\kappa[u]=(\kappa_1,\ldots,\kappa_n )$ of $\Sigma$
at $\big(y_0,u(y_0)\big)$ are the roots of the equation 
$\det\big(a_{ij}-\kappa g_{ij}\big)=0$
and $g_{\alpha\beta}(y_0)=\tilde{g}_{\alpha\beta}(y_0)=\delta_{\alpha\beta}$ for 
$1\leq\alpha,\beta\leq n-1,$ 
they satisfy
\begin{align*}
\det \left(\begin{array}{cccc}
\frac{1}{W}u_{1;1}-\kappa & 0 & \cdots & \frac{1}{W}u_{1;\nu}-g_{1n} \\ 
0 & \frac{1}{W}u_{2;2}-\kappa & \cdots  &  \frac{1}{W}u_{2;\nu}-g_{2n}\\ 
\vdots & \,  & \ddots & \vdots \\ 
\frac{1}{W}u_{\nu;1}-g_{1n}  & \frac{1}{W}u_{\nu;1}-g_{2n}  & \cdots & \frac{1}{W}u_{\nu;\nu}-
\kappa g_{nn}
\end{array} \right)=0.
\end{align*}
Therefore, by Lemma 1.2 of \cite{CNSIII}
the principal curvatures $\kappa[u](y)=(\kappa_1, \ldots ,\kappa_n)$
of $\Sigma,$ at $\big(y_0,u(y_0)\big),$ behave like
\begin{align}
\label{ki}\kappa_\alpha &=\frac{1}{W}u_{\alpha;\alpha}+o(1), \hspace*{2cm}  1\leq\alpha\leq n-1,\\
\label{kn}\kappa_n &=\frac{1}{Wg_{nn}}u_{\nu;\nu}\left(1+O
\left(\frac{1}{u_{\nu;\nu}}\right)\right),
\end{align}
as $|u_{\nu;\nu}|\rightarrow\infty.$ Since $u$ is admissible, we have $\kappa'[u]
=(\kappa_\alpha)\in\Gamma',$ therefore $W\kappa'[u]\in\Gamma'.$
Hence, for $u_{\nu;\nu}$  large we have $\tilde\kappa=( u_{1;1}, \ldots,  u_{n-1;n-1})
\in\Gamma',$ since $\Gamma'$ is open and 
we can assume $u_{\nu;\nu}\geq 0.$ Since $y_0\in\partial\Omega$ is arbitrary,
it follows from the gradient, tangent and tangent-normal second derivative estimates 
previously established that there exists a uniform positive 
constant $N_0>0$ such that the eigenvalues $\tilde\kappa$
of $T_u$ satisfy 
\begin{equation}
\label{N-0}
\tilde\kappa\in\Gamma', \hspace*{.5cm} \textrm{if} \hspace*{.5cm} u_{\nu;\nu}\geq N_0.
\end{equation}
The following lemma is the key ingredient to obtain the double normal boundary estimate. 
It is an adaption of the technique used in \cite{CNSIII} and \cite{GUAN-2},
using the brilliant idea introduced by Trundiger in \cite{TRU-2}.

\begin{lemma}
\label{LONGE-BORDO}
Let $N_0>0$ be the constant defined in (\ref{N-0}) and suppose $u_{\nu;\nu}\geq N_0.$
Then there exists a uniform constant $c_0>0$ such that
\[
d(y)=d(\tilde\kappa[u](y))\geq c_0 \hspace*{.5cm} \textrm{on}\,\, \partial\Omega.
\]
\end{lemma}
\begin{proof}
Consider a point $y_0\in\partial\Omega$ where the function $d(y)$ 
attains its minimum in $\partial\Omega.$
It suffices to prove that $d(y_0)\geq c_0>0.$ As above we fix Fermi coordinates $(y^i)$ 
in $M$ along $\partial\Omega,$ centered at $y_0,$ such that  
$y^n$ is the normal coordinate and the tangent 
coordinate vectors $\{\frac{\partial}{\partial y^\alpha}|_{y_0}\}_{\alpha<n}$ 
diagonalize $T_u$ at $y_0$ with respect to the inner product
given by $\tilde\sigma+\tilde{\nabla}\varphi\otimes
\tilde{\nabla}\varphi.$  We choose indices such that
\[
\tilde\kappa_1(y_0)\leq \cdots\leq \tilde\kappa_{n-1}(y_0).
\]
From (\ref{RELACAO-U-PHI}) the coordinate system $(y^\alpha)$
diagonalizes also the restriction of
$\nabla^2 u$ to $T(\partial\Omega)$ at $y_0$ and
\begin{equation}
\label{RELACAO-U-PHI-2}
\tilde \kappa_{\alpha}(y_0)= u_{\alpha;\alpha}(y_0),  \hspace*{.5cm} \alpha<n.
\end{equation}
We extend $\nu$ to the coordinate neighborhood  
by taking its parallel transport along normal geodesics departing from $\partial\Omega$
and set
\[
b_{\alpha\beta}=\Pi\left( \frac{\partial}{\partial y^\alpha}, \frac{\partial}{\partial y^\beta}
\right)=\left \langle\nabla_{\frac{\partial}{\partial y^\alpha}}\frac{\partial}{\partial y^\beta},
\nu\right\rangle.
\]
Using Lemma 6.1 of \cite{CNSIII}, we find a vector $\gamma'=(\gamma_1,\ldots , \gamma_{n-1})
\in\mathbb{R}^{n-1}$ such that
\begin{align*}
\gamma_1\geq \cdots \geq\gamma_{n-1}\geq 0,\, \hspace*{.5cm}\sum_{\alpha<n}\gamma_\alpha=1
\end{align*}
and
\begin{equation}
\label{D-X-0}
d(y_0)=\sum_{\alpha<n}\gamma_{\alpha}\tilde{\kappa}_\alpha(y_0) =
 \sum_{\alpha<n}\gamma_{\alpha} u_{\alpha;\alpha}(y_0).
\end{equation}
Moreover,
\begin{equation}
\label{CASA-GAMA-LINHA}
\Gamma'\subset \left\{\lambda'\in\mathbb{R}^{n-1}\, :\, \gamma'\cdot \lambda'>0  \right\}.
\end{equation}
Applying Lemma 6.2 of \cite{CNSIII}, with $\gamma_n=0,$ we get for all $y\in\partial
\Omega$ near $y_0$ 
\begin{align}
\label{DES-T}
\sum_{\alpha<n}\gamma_\alpha T_{\alpha\alpha}(y)=\sum_{\alpha<n}\gamma_\alpha 
u_{\alpha;\alpha}(y) \geq \sum_{\alpha<n}\gamma_\alpha \tilde{\kappa}_{\alpha}(y)\geq d(y)
\geq d(y_0),
\end{align}
where we have used (\ref{CASA-GAMA-LINHA}) and $|\gamma|\leq 1$ in the second inequality. 
Differentiating covariantly the equality $u-\varphi=0$ on $\partial\Omega$ we get
\begin{equation}
\label{RELACAO-U-UNDO-U}
(u-\varphi)_{\xi;\eta}=-(u-\varphi)_\nu \Pi(\xi, \eta) \hspace*{.5cm}
\textrm{on}\,\, \partial\Omega,
\end{equation}
for any vectors fields $\xi$ and $\eta$ tangent to $\partial\Omega.$ 
Then, for $y\in\partial\Omega$ near $y_0,$ we have
\begin{align*}
u_\nu(y)\sum_{\alpha<n}\gamma_\alpha b_{\alpha\alpha} (y)&= \sum_{\alpha<n}
\gamma_\alpha(\varphi-u)_{\alpha;\alpha}(y).
\end{align*}
Then
\begin{align}
\label{CONTA-U-PHI}
\begin{split}
u_\nu (y)\sum_{\alpha<n}\gamma_\alpha b_{\alpha\alpha}(y)
&= \sum_{\alpha<n}\gamma_\alpha \varphi_{\alpha;\alpha}(y)
-\sum_{\alpha<n}\gamma_\alpha u_{\alpha;\alpha}(y)
\\ &\leq \sum_{\alpha<n}\gamma_\alpha\varphi_{\alpha;\alpha}(y)-d(y_0),
\end{split}
\end{align}
where we use (\ref{DES-T}) in the last inequality. \\
Since $\underline u$ is locally strictly convex in 
a neighborhood of $\partial\Omega$ it follows that
$\kappa'(\underline u_{\alpha;\beta}(y_0))$  belongs to $\Gamma'$
(since $\Gamma^+\subset\Gamma$). We point out that
$\kappa'(\underline u_{\alpha;\beta})$ denotes the eigenvalues
of $\nabla^2\underline u.$ 
We may assume 
\[
d(y_0)<\frac{1}{2}d(\kappa'(\underline u_{\alpha;\beta}(y_0)),
\]
otherwise we are done. Now we use the equality $ u=\underline u$ on $\partial\Omega$
to get
\begin{align*}
(u-\underline u)_\nu\sum_{\alpha<n}\gamma_\alpha b_{\alpha\alpha}
= \sum_{\alpha<n}\gamma_\alpha(\underline u-u)_{\alpha;\alpha},
\end{align*}
on $\partial\Omega.$ Therefore we conclude from
(\ref{RELACAO-U-UNDO-U}), (\ref{CASA-GAMA-LINHA}) and 
Lemma 6.2 of \cite{CNSIII} that
\begin{align*}
(u-\underline u)_\nu(y_0)\sum_{\alpha<n}\gamma_\alpha b_{\alpha\alpha}(y_0)&=
\sum_{\alpha<n}\gamma_\alpha\underline u_{\alpha;\alpha}(y_0)
-\sum_{\alpha<n}\gamma_\alpha u_{\alpha;\alpha}(y_0)
\\ & \geq d\big(\kappa'(\underline{u}_{\alpha;\beta}(y_0)\big)-d(y_0) 
> \frac{1}{2} d\big(\underline{u}_{\alpha;\beta}(y_0)\big)>0.
\end{align*}
Since $(u-\underline u)_\nu\geq 0$ on
$\partial\Omega,$  there exist uniform positive
constants $ c,\delta>0,$ such that
\[
\sum_{\alpha<n}\gamma_\alpha b_{\alpha\alpha}(y)\geq c>0, 
\]
for every $y\in \Omega$ satisfying $\textrm{dist}(y, y_0)<\delta.$
Hence we may define the function
\begin{align}
\label{CARA-MU}
\mu(y) =\frac{1}{\sum_{\alpha<n}\gamma_\alpha b_{\alpha\alpha}(y)}
\left(\sum_{\alpha<n}\gamma_\alpha\varphi_{\alpha;\alpha}(y)-d(y_0)\right),
\end{align}
for $y\in \Omega_{\delta}=\{ x\in\Omega\, :\, \rho(x)=\textrm{dist}(x,y_0)<\delta\}.$
It follows from (\ref{CONTA-U-PHI}) that $ u_\nu \leq \mu$
on $\partial\Omega\cap\partial\Omega_{\delta}$ 
while (\ref{D-X-0}) and (\ref{RELACAO-U-UNDO-U}) imply
$u_\nu (y_0) = \mu(y_0).$ Now we may proceed as it was done for
the mixed normal-tangential derivatives to get the estimate $u_{\nu;\nu}(y_0)
\leq C,$ for a uniform constant $C.$
In fact,  at the definition of the function $w$ in (\ref{CARA-W}) we can choose
the vector field $\xi$ as being an extension of $\nu$ and change the function $\mu$ there
by the function $\mu$ defined on 
(\ref{CARA-MU}). Defining $\tilde w$ 
in the same way as in (\ref{w-til}), the inequality
(\ref{LW-TIL})  remains true, hence the function $h$
defined at equation (\ref{CARA-H}) still satisfies $L[h]\leq 0$ in $\Omega_{\delta}$
and $h\geq 0$ on $\partial\Omega_{\delta}\cap\Omega,$ for
appropriate constants $a_0,b_0,c_0$ and $\delta>0$ sufficiently
small. To get the inequality $h\geq 0$ on $\partial\Omega_{\delta}\cap
\partial\Omega$ we must  use that $u_\nu \leq \mu$ on 
$\partial\Omega\cap\partial\Omega_{\delta}.$ 
Then, like it was done for 
the mixed normal-tangential derivatives case we get 
\begin{equation}
\label{U-NUNU-X-0}
u_{\nu;\nu}(y_0)
\leq C.
\end{equation}

Therefore  $\kappa[u](y_0)$ is contained in an
{\it a priori} bounded subset of $\Gamma.$ Since
\[
F[u]=f(\kappa[u])=\Psi\geq \Psi_0=\inf \Psi >0
\]
it follows from (\ref{limite}) that
\[
\textrm{dist}(\kappa[u](y_0),\partial\Gamma)\geq c_0>0
\]
for a uniform constant $c_0>0.$ Thus $d(y_0)\geq c_0,$
for a uniform constant $c_0>0.$
\end{proof}
We are now in position to prove (\ref{NORMAL-NORMAL}). 
We assume that
$u_{\nu;\nu}\geq N_0,$ where $N_0$ is the
uniform constant defined above (otherwise we are done). 
By our choice of $N_0$ we have $\tilde\kappa[u]\in\Gamma'$ on $\partial\Omega,$ where
$\tilde\kappa$ are the eigenvalues of the tensor $T_u$
defined in (\ref{DEF-T}).
Fixed $y\in\partial\Omega,$ we choose Fermi coordinates 
centered at $y$ as it was done in (\ref{matrix-second-form}) to conclude that
$\tilde\kappa[u](y)=(u_{1;1}, \ldots , u_{n-1;n-1})$
are the eigenvalues of $T_u$ and the principal curvatures 
$\kappa[u](y)=(\kappa_1, \ldots ,\kappa_n)$
of $\Sigma,$ at $(y,u(y)),$ behave as is described in (\ref{ki}) and (\ref{kn}).
Therefore, since $\frac{1}{W}\tilde\kappa[u]\in\Gamma',$ there exists a uniform constant 
$N_1$ such that if $u_{\nu;\nu}(y)\geq N_1$ then
the distance of $\kappa'[u]=\kappa'\big(a_i^j[u]\big)(y)$ to $\partial\Gamma'$ 
is greater then $c_0/2,$ where $c_0$ is the constant at Lemma \ref{LONGE-BORDO}. Thus
\[
d\Big(\kappa'[u](y)\Big)\geq \frac{c_0}{2},
\]
for $y\in\Lambda=\{y\in\Omega\, :\, u_{\nu;\nu}(y)\geq N_1\}.$
On the other hand, it follows from (\ref{elipticidade}) that there exists a uniform constant $
\delta_0>0$ such that 
\begin{align}
\label{f-infty}
\lim_{t\rightarrow\infty }f(\kappa'[u](y), t)\geq \Psi(y,u)+\delta_0
\end{align}
uniformly for $y\in \Lambda,$ then we have a uniform upper bound
$\kappa_n[u](y) \leq C$ for $y\in \Lambda.$ This yields a uniform upper bound
$u_{\nu;\nu}(y)\leq C$ for $y\in \Lambda$ and establishes (\ref{NORMAL-NORMAL}).


\section{Global bounds for the second derivatives}
\label{section6}

This section is devoted to the proof of the global Hessian estimate. 
We will show that the terms of the second fundamental form $b$ of the graph of $u$ are bounded by
above. Combined with (\ref{gamma-longe-zero})
(see Section \ref{section2}), this provides us with uniform bounds for $b$.
Since we already have the $C^1$ estimate, then this
information allows us to obtain the global second derivative estimate.

\begin{proposition}
\label{estimativa-hessiano}
Suppose that conditions (\ref{elipticidade})-(\ref{produto}) hold and that there exists
a locally strictly convex function $\chi\in C^2(\overline{\Omega}).$
Let $u\in C^4(\Omega)\cap C^2(\overline{\Omega})$ be an admissible solution of 
(\ref{PD-EQ}). Then
\begin{equation}
\label{eq-estimativa-hessiano}
|\nabla^2 u|\leq C \hspace*{.5cm}\textrm{in}\,\,\overline{\Omega},
\end{equation}
where $C$ depends on $|u|_1,$ $\max_{\partial\Omega}|\nabla^2u|,$   
$|\underline{u}|_2$ and other known data.
\end{proposition}
\begin{proof}
First we extend the locally strictly convex function $\chi\in C^2(\overline{\Omega})$ to 
$\overline\Omega\times \mathbb{R}$ by setting 
\[
\chi(x,t)=\chi(x)+t^{2}.
\]
This extension is also locally strictly convex and we use the symbol
$\chi$  also to represent it.
We then  define the following function on the unit
tangent bundle of $\Sigma,$
\[
\tilde{\zeta}(y,\xi)=b(\xi,\xi)\exp \big(\phi(\tau(y))+\beta\chi(y)\big),
\]
where $y\in\Sigma,$ $\xi$ is a unit tangent vector to $\Sigma$ at $y,$ the function 
$\tau$ is the support function defined on $\Sigma$ by $\tau=\langle N, \partial_t\rangle,$ 
$\beta>0$ is a constant to be chosen later and $\phi$ is a real function defined as follows.
The function $\tau$ is bounded by constants 
depending on the bound for  $|\nabla u|.$  Hence, it is possible to choose $a>0$ so that
$\tau\geq 2a$. Thus,  we define
\begin{equation*}
\phi(\tau)=-\ln(\tau-a).
\end{equation*}
Differentiating with respect to $\tau,$ 
\begin{equation}
\label{estimativa-hessiano-conta1}
\ddot\phi-(1+\epsilon)\dot\phi^{2}=\frac{1}{(\tau-a)^{2}}-\frac{1
+\epsilon}{(\tau-a)^{2}}=-\frac{\epsilon}{(\tau-a)^{2}}<0,
\end{equation}
for any positive constant $\epsilon>0.$ Notice that, by  the choice of $a,$ given an arbitrary  
positive constant $C$, we have
\begin{align}
\label{estimativa-hessiano-conta2}
\begin{split}
 -(1+\dot\phi\tau)+C(\ddot\phi-(1+\epsilon)\dot\phi^{2})
=   -1+\frac{\tau}{\tau-a}-\frac{c_{1}\epsilon}{(\tau-a)^{2}}
 \geq \frac{a^{2}}{2(\tau-a)^{2}} \ge \hat C,
 \end{split}
\end{align}
for some positive constant $\hat C$ depending on the bound for $|\nabla u|$.

If the maximum of  $\tilde{\zeta}$ is achieved on $\partial\Sigma,$
we can estimate it in terms of uniform constants (see the last section)
and we are done.  Thus, suppose the maximum of  $\tilde{\zeta}$ is attained at a
point $y_0=\left(x_0,u(x_0)\right)\in\Sigma,$ with $x_0\in\Omega,$ and along 
the direction $\xi_0$ tangent to $\Sigma$ at
$y_0=\left(x_0, u(x_0)\right)$. We fix a normal coordinate system $(y^i)$
of $\Sigma$ centered at $y_0$ such that
\[
\frac{\partial}{\partial y^1}\Big|_{y_0}=\xi_0.
\]
Notice that $\xi_0$ is a principal direction of $\Sigma$ at $y_0$, thus 
$a_{1i}\left(y_0\right)=0,$ for  any $i>1.$
Consider the local function  $a_{11}=b(\frac{\partial}{\partial x^1},
\frac{\partial}{\partial x^1}).$ Then the function
\begin{equation}
\zeta=a_{11}\exp(\phi(\tau)+\beta\chi)
\end{equation}
attains maximum at  $y_0=\left(x_0,u(x_0)\right)$. Hence, it holds at $y_0$
\begin{align}
\label{estimativa-hessiano-conta3}
0 =(\ln
\zeta)_i=\frac{a_{11;i}}{a_{11}}+\dot\phi\tau_i+\beta\chi_i
\end{align}
and the Hessian matrix with components
\begin{align*}
(\ln \zeta)_{i;j}
=\frac{a_{11;ij}}{a_{11}}-\frac{a_{11;i}a_{11;j}}{a^{2}_{11}}
+\dot\phi\tau_{i;j}+\ddot\phi\tau_i\tau_j+\beta\chi_{i;j}
\end{align*}
is negative-definite. Thus
\begin{align}
\begin{split}
\label{estimativa-hessiano-conta4}
G^{ij}(\ln \zeta)_{i;j} =&\frac{1}{a_{11}}G^{ij}a_{11;ij}
-\frac{1}{a^{2}_{11}}G^{ij}a_{11;i}a_{11;j}
+\dot\phi G^{ij}\tau_{i;j}\\ &
+\ddot\phi G^{ij}\tau_i\tau_j +\beta G^{ij}\chi_{i;j}\leq 0.
\end{split}
\end{align}
We may rotate the coordinates $(y^2, \ldots,y^n)$ in such a
way that the new coordinates diagonalize $\{a_{ij}(y_0)\}.$
By Lemma \ref{derivada-general-f}  
$\{G^{ij}\}$ is also diagonal with $G^{ii}=\frac{1}{W}f_{i}$. We denote
$\kappa_{i}=a_{ii}(y_0)$ and choose indices in such a way that
\[
\kappa_{1}\geq\kappa_{2}\geq\cdot\cdot\cdot\geq\kappa_{n}.
\]
Moreover, we assume without loss of generality that $\kappa_{1}>1$
at $y_0$. Thus, according to Lemma \ref{derivada-general-f}, we have
\[
f_{1}\leq f_{2}\leq\cdot\cdot\cdot\leq f_{n}.
\]
From (\ref{estimativa-hessiano-conta4}),
\begin{align}
\label{estimativa-hessiano-conta5}
\sum_i \Big(\frac{1}{\kappa_{1}}f_{i}a_{11;ii}
-\frac{1}{\kappa_{1}^{2}}f_{i}|a_{11;i}|^{2}+\dot\phi
f_{i}\tau_{i;i}+\ddot\phi f_{i}|\tau_i|^{2}  +\beta
f_{i}\chi_{i;i}\Big)\leq 0.
\end{align}
Now, we differentiate covariantly with respect to the metric
$(g_{ij})$ in $\Sigma$ the equation  (\ref{P2-1}) in the
direction of $\frac{\partial}{\partial y^1}|_{y_0}$ obtaining
$F^{ij}a_{ij;1}=\Psi_1$
and differentiating again
\begin{equation}
\label{estimativa-hessiano-conta6}
F^{ij}a_{ij;11} +F^{ij,kl}a_{ij;1}a_{kl;1}=\Psi_{1;1}.
\end{equation}
From the Simons formula (\ref{SIMONS-FORMULA}) we have
\begin{align}
\begin{split}
\label{estimativa-hessiano-conta7}
F^{ij}a_{ij;11}  = F^{ii}a_{ii;11}=\sum_{i}\big(f_{i}a_{11;ii}+\kappa_{1}
f_{i}\kappa_{i}^{2}-\kappa_{1}^{2} f_{i}\kappa_{i} \\
+\kappa_{1} f_{i}\bar R_{i0i0} -\bar R_{1010}
f_{i}\kappa_{i}+f_{i}\bar R_{i1i0;1}-f_{i}\bar R_{1i10;i}\big).
\end{split}
\end{align}
It follows from $c_0\leq\sum_{i}f_{i}\lambda_{i}\leq f=\Psi$ that
\begin{align*}
\begin{split}
 F^{ij}a_{ij;11} \leq & -\kappa_{1}^{2}c_0+|\bar
R_{1010}|\Psi \\ &+\sum_{i}\big(f_{i}a_{11;ii}+\kappa_{1}f_{i}\kappa_{i}^{2}
+\kappa_{1}f_{i}\bar R_{i0i0}
 +  f_{i}\bar R_{i0i0;1}-f_{i}\bar R_{1010;i}\big).
 \end{split}
\end{align*}
Combining this expression and  (\ref{estimativa-hessiano-conta6}) we obtain
\begin{align*}
\begin{split}
\sum_{i}f_{i}a_{11;ii} \geq &\Psi_{1;1}-F^{ij,kl}a_{ij;1}a_{kl;1}
+\kappa_1^2 c_0-|\bar
R_{1010}|\psi\\
& - \sum_i\big(\lambda_{1}f_{i}\lambda_{i}^{2}
-\lambda_{1}f_{i}\bar R_{i0i0} -  f_{i}\bar R_{i0i0;1}+f_{i}\bar
R_{1010;i}\big).
\end{split}
\end{align*}
Replacing this equation into (\ref{estimativa-hessiano-conta5}) we get
\begin{align*}
\begin{split}
&\frac{1}{\kappa_{1}}\big(\Psi_{1;1}-F^{ij,kl}a_{ij;1}a_{kl;1}
+\kappa_1^2c_0 -|\bar R_{1010}|\Psi\big) 
\\ &-\frac{1}{\kappa_{1}}\sum_i\big(\kappa_{1}f_{i}\kappa_{i}^{2}
-\kappa_{1}f_{i}\bar R_{i0i0} 
 - f_{i}\bar R_{i0i0;1}+f_{i}\bar R_{1010;i}\big) \\ 
&+\sum_i\Big(\dot\phi f_{i}\tau_{i;i}-\frac{1}{\kappa_{1}^{2}}f_{i}|a_{11;i}|^{2}
+\ddot\phi f_{i}|\tau_i|^{2}
+ \beta  f_{i}\chi_{i;i}\Big)\leq 0.
\end{split}
\end{align*}
Therefore,
\begin{align*}
\frac{\Psi_{1;1}}{\kappa_{1}}&+\frac{1}{\kappa_{1}}\big(c_0\kappa_{1}^{2}-\Psi
|\bar R_{1010}|\big) -\frac{1}{\kappa_{1}}F^{ij,kl}a_{ij;1}a_{kl;1}
-\sum_{i}f_{i}\kappa_{i}^{2}\\
&-\sum_{i}f_{i}\bar R_{i0i0}  +\sum_i\Big(\dot\phi f_{i}\tau_{i;i}-\frac{1}{\kappa_{1}^{2}}f_{i}|a_{11;i}|^{2}
+\ddot\phi f_{i}|\tau_i|^{2}
+\beta f_{i}\chi_{i;i}\Big)\\ &
-\frac{1}{\kappa_{1}}\sum_{i}f_{i}\big(\bar R_{i0i0;1}-\bar
R_{1010;i}\big) \leq 0.
\end{align*}
It is well known that
\begin{align*}
\tau_i &=-a_i^k\eta_k \\ \tau_{i;j} &= -\eta^k a_{ii;k}-\eta^k\bar R_{kij0}-\tau a_i^ka_{kj},
\end{align*}
where $\eta^k$ are the components of the vector $\partial_t^T,$ which
is the projection of $\partial_t$ onto $T\Sigma.$
Hence, 
\begin{align*}
\dot\phi\sum_{i}f_{i}\tau_{i;i} =
-\dot\phi\Big(\sum_{i}\eta^k f_{i}a_{ii;k}
+\sum_{i}\eta^k\bar R_{kii0} f_{i}\Big)-\dot\phi\tau
\sum_{i}f_{i}\kappa_{i}^{2}.
\end{align*}
From $\sum_i f_ia_{ii;k}= \Psi_k,$
\begin{align*}
\dot\phi\sum_{i}f_{i}\tau_{i;i} = -\dot \phi\Big(\eta^k \Psi_k +\sum_{i}\eta^k \bar
R_{kii0}f_{i}\Big)-\dot\phi \tau
\sum_{i}f_{i}\kappa_{i}^{2}.
\end{align*}
We denote by $T= \sum_i f_i.$ By estimating the ambient curvature terms,
\[
\sum_{i}\eta^k\bar R_{kii0}f_{i}\leq CT.
\]
Then,
\begin{align*}
-\dot \phi\Big(\eta^k \Psi_k +\sum_{i}\eta^k \bar
R_{kii0}f_{i}\Big) \geq
-|\dot\phi|(C+CT).
\end{align*}
Therefore
\[
\dot \phi \sum_{i}f_{i}\tau_{i;i} \geq
-|\dot\phi|(C+CT)-\dot\phi\tau \sum_{i}f_{i}\kappa_{i}^{2}.
\]
Now, we suppose without loss of generality that
\begin{align*}
\kappa_{1} \geq\frac{1}{C}\sum_{i}|R_{i0i0;1}-R_{1010;i}|,\hspace*{.5cm}
-\frac{1}{\kappa_1}\Psi |\bar R_{1010}| \geq-C, \hspace*{.5cm} \textrm{and}
\hspace*{.5cm} \frac{\Psi_{1;1}}{\kappa_1}  \geq -C
\end{align*}
for a positive constant $C>0$. Since 
\[
\Psi_{1;1}=\Psi_{t;t}(u_1)^2+\Psi_tu_{1;1}+\Psi_{1;1},
\]
the above assumptions is allowed. Finally,
\[
-\sum_i f_i \bar R_{i0i0}\geq -T \max_i |\bar R_{i0i0}|\geq -CT.
\]
We conclude from these inequalities  that
\begin{align}
\label{estimativa-hessiano-conta8}
\begin{split}
-C-CT+c_0\kappa_{1}-\frac{1}{\kappa_{1}}F^{ij,kl}a_{ij;1}a_{kl;1}
-\sum_{i}f_{i}\kappa_{i}^{2}
- \frac{1}{\kappa_{1}^{2}}\sum_{i}f_{i}|a_{11;i}|^{2}\\
-|\dot\phi|(C+CT) -\dot\phi\tau \sum_{i}f_{i}\kappa_{i}^{2}
+\ddot\phi\sum_{i}f_{i}|\tau_i|^{2}+ \beta\sum_i f_i \chi_{i;i}\leq 0.
\end{split}
\end{align}

Now, to proceed further with our analysis we consider two cases.\\
{\it Case I:} In this case we suppose that
$\kappa_{n}\leq-\theta\kappa_{1}$ for some positive constant
$\theta$ to be chosen later.

Using (\ref{estimativa-hessiano-conta3}) 
and the Cauchy inequality we get
\begin{align}
\label{estimativa-hessiano-conta9} 
\frac{1}{\kappa_{1}^{2}}f_{i}|a_{11;i}|^{2}
=f_{i}|\dot\phi\tau_i+\beta\chi_i|^{2}\leq
(1+\frac{1}{\epsilon})\beta^{2}f_{i}|\chi_i|^{2}
+(1+\epsilon)\dot\phi^{2}f_{i}|\tau_i|^{2},
\end{align}
for any $\epsilon>0$ and any $1\leq i\leq n.$ Now we replace
the sum of the terms in (\ref{estimativa-hessiano-conta9}) in the
inequality  (\ref{estimativa-hessiano-conta8}) to obtain
\begin{align*}
\begin{split}
c_0\kappa_{1}-C(1+|\dot\phi|)-CT(1+|\dot\phi|)-\frac{1}{\kappa_{1}}F^{ij,kl}a_{ij;1}a_{kl;1}
-(1+\dot\phi\tau) \sum_{i}f_{i}\kappa_{i}^{2}\\
- (1+\frac{1}{\epsilon})\beta^2\sum_if_{i}|\chi_i|^{2}
+\big(\ddot\phi-(1+\epsilon)\dot\phi^{2}
\big)\sum_{i}f_{i}|\tau_i|^{2}+ \beta\sum_i f_i \chi_{i;i}\leq 0.
\end{split}
\end{align*}
Since $\{a_{ij}\}$ is diagonal
at $y_0,$ 
\[
\sum_{i}f_{i}|\tau_i|^{2}=\sum_{i}f_{i}\lambda_{i}^{2}|\eta_i|^{2}\leq
C\sum_{i}f_{i}\kappa_{i}^{2},
\]
so, it follows from (\ref{estimativa-hessiano-conta1}) that
\[
\big( \ddot\phi-(1+\epsilon)\dot\phi^{2}\big)
\sum_{i}f_{i}|\tau_i|^{2}\geq  \big( \ddot\phi-(1+\epsilon)\dot\phi^{2}\big)
C\sum_i f_i\kappa_i^2.
\]
Since $|D\chi|$ is a known data we have $\sum_if_{i}|\chi_i|^{2}\leq CT.$
Hence,
\begin{align}
\begin{split}
\label{estimativa-hessiano-conta10} 
c_0\kappa_{1}-C(1+|\dot\phi|)-\frac{1}{\kappa_{1}}F^{ij,kl}a_{ij;1}a_{kl;1}
-\big(1+|\dot\phi|+ (1+\frac{1}{\epsilon})\beta^{2}\big)CT\\
+\Big(-(1+\dot\phi\tau)+C\big( \ddot\phi-(1+\epsilon)\dot\phi^{2}\big)
\Big)\sum_i f_i\kappa_i^2+ \beta\sum_i f_i \chi_{i;i}\leq 0.
\end{split}
\end{align}
Using the concavity of $F$ and the convexity of $\chi$
 we may discard the third and the last terms in the
left-hand side of (\ref{estimativa-hessiano-conta10}) since they are nonnegative, obtaining
\[
-C_1(\beta) -C_{2}(\beta)T+c_0\kappa_1+\hat
C\sum_{i}f_{i}\kappa_{i}^{2}\leq 0,
\]
where $C_1$ depends linearly on $\beta$ and $C_2$ depends
quadratically on $\beta$.  Since $f_{n}\geq\frac{1}{n}T$, we have
\[
\sum_{i}f_{i}\kappa_{i}^{2}\geq f_{n}\kappa_{n}^{2}\geq
\frac{1}{n}\theta^{2}T\kappa_{1}^{2}.
\]
Thus it follows that
\begin{align}
\label{estimativa-hessiano-conta11} 
-C_1 - C_2 T + c_0\kappa_{1}+\hat
C\frac{1}{n}\theta^{2}T\kappa_{1}^{2}\leq 0.
\end{align}
This inequality shows that  $\kappa_{1}$ has a uniform upper bound. In fact, notice that the coefficients
of the terms in $T$ in (\ref{estimativa-hessiano-conta11}) are
\[
\hat C\frac{1}{n}\theta^{2}\kappa_{1}^{2}-C_2.
\]
Then, if $\kappa_1\geq \bar C$ for a (suitable) uniform constant $\bar C,$ we have
\[
\hat C\frac{1}{n}\theta^{2}\kappa_{1}^{2}-C_2\geq 0.
\]
In this case, since $T=\sum_i f_i\geq 0,$ we 
may discard the terms in $T$ in (\ref{estimativa-hessiano-conta11}) to obtain
$-C_1+c_0\kappa_1\leq 0,$ i.e.,
\[
\kappa_1\leq \frac{C_1}{c_0}.
\]
{\it Case II:} In this case we assume that
$\kappa_{n}\geq-\theta\kappa_{1}$. Hence,
$\kappa_{i}\geq-\theta\kappa_{1}$. We then group the indices 
$\{1,...,n\}$ in two sets 
\begin{align*}
I_1&=\{j;f_{j}\leq4f_{1}\}, \\
I_2&=\{j;f_{j}>4f_{1}\}.
\end{align*} 
Using (\ref{estimativa-hessiano-conta9}), we have for $i\in I_1$
\begin{align*}
\frac{1}{\kappa_{1}^{2}}f_{i}|a_{11;i}|^{2}
&\leq (1+\epsilon)\dot\phi^{2}f_{i}|\tau_i|^{2}
+(1+\frac{1}{\epsilon})(\beta)^{2}f_{i}|\chi_i|^{2}\\
&\leq(1+\epsilon)\dot\phi^{2}f_{i}|\tau_i|^{2}
+C(1+\frac{1}{\epsilon})(\beta)^{2}f_{1}.
\end{align*}
Therefore, it follows from (\ref{estimativa-hessiano-conta9}) that
\begin{align*}
 &-C-CT +c_0\kappa_{1}
-\frac{1}{\lambda_{1}}F^{ij,kl}a_{ij;1}a_{kl;1}
-\big(1+\dot\phi\tau
\big)\sum_{i}f_{i}\kappa_{i}^{2} 
 - \frac{1}{\kappa_{1}^{2}}\sum_{j\in
I_2}f_{j}|a_{11;j}|^{2} \\ & -|\dot\phi|(C+CT)
 +\big(\ddot\phi-(1+\epsilon)\dot\phi^{2}\big)\sum_{i}f_{i}|\tau_i|^{2} 
 -  C(1+\frac{1}{\epsilon})\beta^{2}f_{1}
 +\beta\sum_i f_i\chi_{i;i}\leq 0.
\end{align*}
Notice that we had summed up to the inequality the non-positive
terms
\[
-(1+\epsilon)|\dot\phi|^2\sum_{i\in I_2} f_i |\tau_i|^2.
\]
Using that $|\tau_i| =|\kappa_i \eta_i|\le C\kappa_i$ we may conclude as above that
\begin{equation}
\label{estimativa-hessiano-conta12} 
-\big(1+\dot\phi\tau\big)\sum_{i}f_{i}\kappa_{i}^{2}
+\big(\ddot\phi-(1+\epsilon)\dot\phi^{2}\big)\sum_{i}f_{i}|\tau_i|^{2}\geq
\hat C\sum_{i}f_{i}\kappa_{i}^{2}
\end{equation}
for some positive constant $\hat C>0$. Thus
\begin{align}
\begin{split}
\label{estimativa-hessiano-conta12*}
& -C-CT +c_0\kappa_{1}
-\frac{1}{\kappa_{1}}F^{ij,kl}a_{ij;1}a_{kl;1} +\hat
C\sum_{i}f_{i}\kappa_{i}^{2}
\\&- \frac{1}{\kappa_{1}^{2}}\sum_{j\in
I_2}f_{j}|a_{11;j}|^{2}-|\dot\phi|(C+CT)
-C\big(1+\frac{1}{\epsilon}\big)\beta^{2}f_{1}+\beta\sum_i f_i\chi_{i;i}\leq 0.
\end{split}
\end{align}
Using the Codazzi equation $a_{1j;1}=a_{11;j}+\bar R_{01j1}$
and Lemma \ref{derivada-general-f} we get
\begin{align*}
 -\frac{1}{\kappa_{1}}F^{ij,kl}a_{ij;1}a_{kl;1} &= -\frac{1}{\kappa_{1}}
 \sum_{k,l}f_{kl}a_{kk;1}a_{ll;1}-\frac{1}{\kappa_{1}}\sum_{k\neq
l} \frac{f_k-f_l}{\kappa_k-\kappa_l}\eta_{kl}^2\\ &
\geq -\frac{2}{\kappa_{1}}\sum_{j\in
I_2}\frac{f_{1}-f_{j}}{\kappa_{1}-\kappa_{j}}(a_{1j;1})^{2} \\& =
-\frac{2}{\kappa_{1}}\sum_{j\in
I_2}\frac{f_{1}-f_{j}}{\kappa_{1}-\kappa_{j}}\big(a_{11;j} +\bar
R_{01j1}\big)^{2},
\end{align*}
since $1\notin I_2$ and $\frac{f_k-f_l}{\kappa_k-\kappa_l}\leq 0.$
We claim that for all $j\in I_2$ it holds the inequality
\begin{align}
\label{estimativa-hessiano-conta13} 
-\frac{2}{\kappa_{1}}\frac{f_{1}-f_{j}}{\kappa_{1}
-\kappa_{j}}\geq\frac{f_{j}}{\kappa_{1}^{2}}.
\end{align}
This is equivalent to
\[
2f_{1}\kappa_{1}\leq f_{j}\kappa_{1}+f_{j}\kappa_{j}.
\]
It is clear that $j\in I_2$ implies $f_{j}>4f_{1}.$ If
$\kappa_{j}\geq0$, this is obvious. If $\kappa_{j}<0$, then
$-\theta\kappa_{1}\leq\kappa_{j}<0,$ and then
\[
f_{j}\kappa_{1}+f_{j}\kappa_{j}\geq(1-\theta)f_{j}\kappa_{1} \geq
4(1-\theta)f_{1}\kappa_{1}\geq2f_{1}\kappa_{1},
\]
if we choose  $\theta=1/2$. Hence, with this choice,
we can use (\ref{estimativa-hessiano-conta13}) to obtain
\begin{align*}
-\frac{1}{\kappa_{1}}F^{ij,kl}a_{ij;1}a_{kl;1}& \geq \sum_{j\in I_2}\frac{f_{j}}{\kappa_{1}^{2}}
\big(a_{11;j} +\bar R_{01j1}\big)^{2} \\&= \sum_{j\in I_2}\frac{f_{j}}{\kappa_{1}^{2}}(a_{11;j})^{2}
+2\sum_{j\in I_2}\frac{f_{j}}{\kappa_{1}^{2}}a_{11;j}\bar R_{01j1}
+\sum_{j\in I_2}\frac{f_{j}}{\kappa_{1}^{2}}(\bar R_{01j1})^{2}.
\end{align*}
Using this inequality in (\ref{estimativa-hessiano-conta12*}) and estimating the 
curvature term $(R_{01j1})^2$ we obtain
\begin{align*}
 &-C-CT +c_0\kappa_{1}
+\sum_{j\in I_2}\frac{f_{j}}{\kappa_{1}^{2}}(a_{11;j})^{2}
+2\sum_{j\in I_2}\frac{f_{j}}{\kappa_{1}^{2}}a_{11;j}\bar R_{01j1} +\hat
C\sum_{i}f_{i}\kappa_{i}^{2}
\\&- \frac{1}{\kappa_{1}^{2}}\sum_{j\in
I_2}f_{j}|a_{11;j}|^{2}-|\dot\phi|(C+CT)
-C\big(1+\frac{1}{\epsilon}\big)\beta^{2}f_{1}+\beta\sum_i f_i\chi_{i;i}\leq 0.
\end{align*}
From (\ref{estimativa-hessiano-conta3}),
\begin{align*}
 &-C-CT +c_0\kappa_{1}
-2\sum_{j\in I_2}\frac{f_{j}}{\kappa_{1}}(\dot\phi\tau_j+\beta\chi_j)
\bar R_{01j1} +\hat C\sum_{i}f_{i}\kappa_{i}^{2}
\\&-|\dot\phi|(C+CT)
-C\big(1+\frac{1}{\epsilon}\big)\beta^{2}f_{1}+\beta\sum_i f_i\chi_{i;i}\leq 0.
\end{align*}
Since $\dot\phi<0,$  $\kappa_j\leq
\kappa_1$ and $-\kappa_j\leq \theta\kappa_1<\kappa_1$
we have 
\begin{align*}
 2\frac{f_{j}}{\kappa_{1}}(-\dot\phi\tau_j)\bar R_{01j1}=
2\frac{f_{j}}{\kappa_1}\dot\phi\kappa_j\eta_j\bar R_{01j1}\geq
2\frac{f_{j}}{\kappa_1}\dot\phi|\kappa_j||\eta_j\bar R_{01j1}|\geq
2f_{j}\dot\phi|\eta_j\bar R_{01j1}|.
\end{align*}
We also suppose, without loss of generality, that
\[
\kappa_1\geq \frac{3| \chi_j\bar R_{01j1}|}{\gamma_0}
\]
for all $j\in I_2,$ where $\gamma_0$ is a positive constant that
satisfies
\[
\chi_{i;i}\geq \gamma_0>0, \hspace*{.5cm}\forall\,\, 1\leq i\leq n.
\]
Note that this assumption is equivalent to
\[
 \frac{\gamma_0}{3}\ge \frac{|\chi_j \bar R_{01j1}|}{\kappa_1},
\]
which implies
\begin{align*}
-2\sum_{j\in I_2} \frac{f_j}{\kappa_1}\beta\chi_j \bar R_{01j1} 
\geq &-2\sum_{j\in I_2} \frac{f_j}{\kappa_1}\beta|\chi_j \bar R_{01j1}|
\geq -2 \sum_{j\in I_2}\frac{\beta f_j \gamma_0}{3}\geq -2 \frac{\beta \gamma_0}{3}T.
\end{align*}
These inequalities imply that
\begin{align*}
 &-C-CT +c_0\kappa_{1}+2\sum_{j\in I_2}f_{j}\dot\phi|\eta_j\bar R_{01j1}|
 -2 \frac{\beta \gamma_0}{3}T \\ &+\hat C\sum_{i}f_{i}\kappa_{i}^{2} 
-|\dot\phi|(C+CT)-C\big(1+\frac{1}{\epsilon}\big)\beta^{2}f_{1}+
\beta\sum_i f_i\chi_{i;i}\leq 0.
\end{align*}
Since $\sum_{j\in I_2}f_{j}\leq T$, $|\eta_j\bar R_{j1}|\leq C$ and
$\dot\phi<0$ we have
\[
-C- \big(C+C|\dot\phi| + 2\beta\frac{\gamma_0}{3}-\beta \gamma_0\big)
T-C\big(1+\frac{1}{\epsilon}\big)\beta^{2}f_{1}+c_0\kappa_{1}
+\hat C f_{1}\kappa_{1}^{2}\leq 0.
\]
Choosing $\beta>0$ sufficiently large, the term in $T$ is positive
and we may discard it, obtaining
\begin{equation}
\label{estimativa-hessiano-conta14} 
-C-C_2(\beta)f_1+c_0\kappa_{1}+\hat C f_{1}\kappa_{1}^{2}\le 0,
\end{equation}
where $C_2$ depends quadratically on $\beta$. Reasoning as above, 
we conclude that this inequality gives an upper bound for $\kappa_1.$
\end{proof}


\section{Proof of the Existence --The Continuity Method}
\label{section7}

In this section we complete the proof of Theorems \ref{teorema1} and 
\ref{teorema2} by the continuity 
method with the aid of the {\it a priori} estimates established previously. 
We apply the continuity method to the family of problems
\begin{align}
\label{continuiti-metodo}
\begin{split}
\left\{\begin{array}{cc}F[u] =t\Psi+(1-t)F[\chi_0] & 
\hspace*{.5cm}\textrm{in}\, \, \Omega \\ 
u=t\varphi +(1-t)\chi_0 & \hspace*{.5cm}\textrm{on}\, \, \partial \Omega,
\end{array} \right.
\end{split}
\end{align}
for $0\leq t\leq 1,$ where $\chi_0$ is a multiple of the locally strictly convex function
$\chi\in C^{2}(\overline\Omega)$ that satisfy
\[
0<F[\chi_0]\leq \Psi.
\]
Clearly all our preceding estimates are independent of the parameter
$t,$ so that under the hypotheses of Theorem \ref{teorema1} and \ref{teorema2},
we conclude an {\it a priori} estimate of the form
\[
|u|_{2,\alpha}\leq C
\]
with constant $C$ depending on $n, \Omega, \underline u, \chi, \varphi $
and $\Psi,$ and hence the unique solvability of the Dirichlet problems 
(\ref{continuiti-metodo}), for all $0\leq t\leq 1,$ then follows.

\subsection*{Acknowledgment}

We want to thank Professor Joel Spruck who made valuable comments and point out some minor 
mistakes in a previous version of this paper during the XVII Brazilian School of Differential Geometry.


\end{document}